\documentclass{amsart}

\usepackage{xparse}
\usepackage{hyperref}
\usepackage{amsmath}
\usepackage{amssymb}
\usepackage{amsthm}
\usepackage{mathrsfs}
\usepackage{tikz-cd}
\usepackage{enumerate}

\DeclareMathOperator{\Spec}{Spec}
\DeclareMathOperator{\Spa}{Spa}
\DeclareMathOperator{\Spd}{Spd}
\DeclareMathOperator{\Hom}{Hom}
\DeclareMathOperator{\id}{id}
\DeclareMathOperator{\supp}{supp}

\def\textdef{\textit}
\def\lt{<}
\def\gt{>}

\theoremstyle{plain}
\newtheorem{theorem}{Theorem}[section]
\newtheorem{proposition}[theorem]{Proposition}
\newtheorem{lemma}[theorem]{Lemma}
\newtheorem{corollary}[theorem]{Corollary}
\theoremstyle{definition}
\newtheorem{definition}[theorem]{Definition}
\newtheorem{remark}[theorem]{Remark}
\newtheorem{example}[theorem]{Example}

\def\rrrarrows{\raisebox{0.167em}{$\rightrightarrows$}\hspace{-1em}\raisebox{-0.167em}{$\rightrightarrows$}}
\def\vect{\mathsf{Vect}}
\def\novisoc{\mathsf{NovIsoc}}
\def\ldparen{(\!(}
\def\rdparen{)\!)}
\def\dparenmult{\ldparen t_1^\novexp, \dotsc, t_n^\novexp \rdparen}
\newcommand{\novr}[1]{{\def\novexp{\mathbb{R}} #1 }}

\NewDocumentCommand{\novdd}{r()O{\novexp}}{
  \mathsf{NovDD}(#1, #2)
}

\NewDocumentCommand{\dparen}{r[]oo}{
  \IfValueTF{#3}{
    \ldparen #1^{\novexp}, #2^{\novexp}, #3^{\novexp} \rdparen
  }{
    \IfValueTF{#2}{
      \ldparen #1^{\novexp}, #2^{\novexp} \rdparen
    }{
      \ldparen #1^{\novexp} \rdparen
    }
  }
}
\NewDocumentCommand{\dbrack}{r[]oo}{
  \IfValueTF{#3}{
    [[ #1^{\novexp}, #2^{\novexp}, #3^{\novexp} ]]
  }{
    \IfValueTF{#2}{
      [[ #1^{\novexp}, #2^{\novexp} ]]
    }{
      [[ #1^{\novexp} ]]
    }
  }
}

\title{Descending finite projective modules from a Novikov ring}
\author{Dongryul Kim}
\date{\today}
\email{dkim04@stanford.edu}
\address{Department of Mathematics, Stanford University, 450 Jane Stanford Way
(Building 380), Stanford, California, USA}

\begin{document}

\maketitle

\begin{abstract}
  We prove a descent result for finite projective modules, motivated by a
  question in perfectoid geometry. Given a commutative ring $A$, we formulate a
  descent problem for descending a finite projective module over the Novikov
  ring with coefficients in $A$ to a finite projective module over $A$. The main
  theorem of this paper is that all such descent data are effective. As an
  application, we prove for every perfect $\mathbb{F}_p$-algebra $A$, a vector
  bundle on $\operatorname{Spd} A$ always descends to a vector bundle on
  $\operatorname{Spec} A$.
\end{abstract}

\tableofcontents

\section{Introduction}
\def\novexp{\Gamma}

The goal of this paper is to prove the following descent result. We refer the
reader to Section~\ref{Sec:Novikov} for the definition of the rings and maps
involved.

\begin{theorem}[Theorem~\ref{Thm:NovDescent}] \label{Thm:Main}
  Let $A$ be a commutative ring and let $\novexp \subseteq \mathbb{R}$ be an
  additive submonoid. Then the category of descent data of finite projective
  modules along the semi-cosimplicial diagram
  \[
    A\dparen[t] \rightrightarrows A\dparen[t][u] \rrrarrows A\dparen[t][u][v]
  \]
  is equivalent to the category of finite projective $A$-modules.
\end{theorem}

\begin{remark}
  In general, the ring $A\dparen[t][u]$ is much larger than the tensor product
  $A\dparen[t] \otimes_A A\dparen[t]$, which is usually where the isomorphism of
  the descent datum is defined over. As a consequence, standard descent results
  for finite projective modules are inapplicable to this setting.
\end{remark}

\begin{remark} \label{Rem:AnalyticToDiscrete}
  \def\novexp{}
  This may heuristically be interpreted as an analytic-to-discrete descent
  result for adic spaces. For example, let $k$ be a field, consider the field
  $k\dparen[t]$ of Laurent series with the $t$-adic topology, and take its adic
  spectrum $X = \Spa(k\dparen[t], k\dbrack[t])$. Then we have
  \[
    Y = X \times_{\Spa(k, k)} X = \bigcup_{n \ge 1} \Spa \Bigl( k\Bigl[t, u,
    \frac{t^n}{u}, \frac{u^n}{t}\Bigr]^\wedge \Bigl[\frac{1}{tu}\Bigr],
    k\Bigl[t, u, \frac{t^n}{u}, \frac{u^n}{t}\Bigr]^\wedge \Bigr)
  \]
  in the category of adic spaces, where we use the $(tu)$-adic topology for the
  completion. The adic space $Y$ can also be regarded as the punctured open unit
  disc over $k\dparen[t]$, and we have $H^0(Y, \mathscr{O}_Y) = k\dparen[t][u]$.
  Then Theorem~\ref{Thm:Main} implies that the category of vector bundles on $X$
  equipped with descent data over $Y$ is equivalent to the category of vector
  bundles on $\Spa (k, k)$.
\end{remark}

The theorem is motivated by question in perfectoid geometry, namely that of
descending a vector bundle from $\Spd A$ to a vector bundle on $\Spec A$, where
$A$ is a discrete perfect ring in characteristic $p$. Using
Theorem~\ref{Thm:Main} with $\Gamma = \mathbb{Z}[p^{-1}]$, we deduce the
following result as an immediate consequence. We refer the reader to
Section~\ref{Sec:Perfectoid} for the definitions of the objects involved.

\begin{theorem}[Corollary~\ref{Cor:VectorBundleOnSpd}, \ref{Cor:FiniteEtaleOverSpd}]
  \label{Thm:MainPerfectoid}
  Let $A$ be a perfect $\mathbb{F}_p$-algebra. Then the natural functor
  \[
    \lbrace \text{vector bundles on } \Spec A \rbrace \to \lbrace \text{vector
    bundles on } \Spd A \rbrace
  \]
  is an exact tensor equivalence. Similarly, the natural functor
  \[
    \lbrace \text{schemes finite \'{e}tale over } \Spec A \rbrace \to \lbrace
    \text{v-sheaves finite \'{e}tale over } \Spd A \rbrace
  \]
  is an equivalence.
\end{theorem}

\begin{remark} \label{Rem:EtaleDescentFails}
  It is not true in general that a v-sheaf \'{e}tale over $\Spd A$ descends to a
  scheme \'{e}tale over $\Spec A$. Consider the following example from
  \cite[Example~18.2.1]{SW20}. Take the perfect $\mathbb{F}_p$-algebra $A =
  \mathbb{F}_p[t^{1/p^\infty}]$, which has associated v-sheaf $(\Spd A)(R, R^+)
  = R^+$ that admits a v-subsheaf $X(R, R^+) =
  R^{\circ\circ}$. Then $X \to \Spd A$ is an open embedding, hence \'{e}tale,
  but does not correspond to an \'{e}tale $A$-scheme. Indeed, for every closed
  point $x \in \Spec A$, the fiber $\Spd k(x) \times_{\Spd A} X$ is empty when
  $x \neq 0$ and is $\Spd \mathbb{F}_p$ when $x = 0$.
\end{remark}

\begin{remark}
  The first part of the theorem is false when we replace $\Spec A$ with an
  affinoid rigid analytic variety over a perfectoid field of characteristic
  zero, by the main result of \cite{Heu22}. On the other hand, the second part
  of the theorem remains true by \cite[Lemma~15.6]{Sch22}.
\end{remark}

The second part of Theorem~\ref{Thm:MainPerfectoid} is
used in \cite{DKvHZ24} to construct integral models of Hodge type
Shimura varieties at parahoric level. Specifically, it allows one to construct a
finite \'{e}tale cover of the special fiber of a Shimura variety from the finite
\'{e}tale cover of the associated v-sheaf.

\begin{remark} \label{Rem:Gleason}
  Ian Gleason announced via private communication an alternative proof of
  Corollary~\ref{Cor:FiniteEtaleOverSpd}, i.e., the finite \'{e}tale part of
  Theorem~\ref{Thm:MainPerfectoid}, which works by comparing the category
  of \'{e}tale $\mathbb{F}_\ell$-local systems on $\Spd A$ and $\Spec A$.
\end{remark}

\subsection{Outline of proof} \label{Sec:Outline}

We sketch the main ideas that go into the proof of Theorem~\ref{Thm:Main}.
In Section~\ref{Sec:Isocrystal}, we introduce the notion of a Novikov
isocrystal, which is somewhat easier to control than a descent datum as it involves
only one variable. In Subsection~\ref{Subsec:DDToIsoc}, we construct a functor
from the category of Novikov descent data for $\Gamma = \mathbb{R}$ to the
category of Novikov isocrystals, and prove that this functor is always fully
faithful. This allows us to reduce the effectivity of a Novikov descent datum
into a similar statement about the associated Novikov isocrystal.

In Subsection~\ref{Subsec:Reduced}, we prove effectivity of Novikov descent data
when $A$ is reduced and $\Gamma = \mathbb{R}$. We first take care the case when
$A$ is an algebraically closed field by using the structure theory of Novikov
isocrystals developed in Section~\ref{Sec:Isocrystal}. Next, we check that when
$A$ is an arbitrary product of algebraically closed fields, effectivity over $A$
from that over each field. This allows to deduce effectivity for a general
reduced ring $A$ by embedding $A$ into the product of the algebraic closures of
its residue fields.

In Subsection~\ref{Subsec:General}, we remove the hypothesis that $A$ is
reduced. We first check that effectivity of a Novikov descent datum can be
extended along a finite-order nilpotent thickening. On the other hand, it turns
out that we can reduce an arbitrary nilpotent thickening to a finite-order
thickening by studying the associated Novikov isocrystal. Combining the two
allows us to prove effectivity in the case of a general ring $A$ with $\Gamma =
\mathbb{R}$. Finally, in Subsection~\ref{Subsec:Exponent}, we remove the
hypothesis that $\Gamma = \mathbb{R}$ and prove Theorem~\ref{Thm:Main} in
its full generality.

\subsection{Acknowledgements}

I would like to thank Patrick Daniels, Pol van Hoften, and Mingjia Zhang for
help formulating the problem, which arose in our work on integral models of
Shimura varieties, as well as insightful conversations. I would also like to
thank my advisor Richard Taylor for his encouragements and diligence during our
meetings, where we carefully went through parts of the argument. I also would
like to thank Brian Conrad, Ian Gleason, Kiran Kedlaya, and the anonymous
referee for helpful comments and suggestions.

\section{Novikov rings} \label{Sec:Novikov}
\def\novexp{\Gamma}

\begin{definition} \label{Def:NovikovOneVar}
  Let $\Gamma \subseteq \mathbb{R}$ be an additive submonoid, and let $A$ be a
  commutative ring. We define the \textdef{Novikov ring} with exponents in
  $\Gamma$ and coefficients in $A$ to be the $A$-algebra
  \[
    A\dparen[t] = \biggl\lbrace \sum_{i \in I} a_i t^{d_i} : a_i \in A, d_i \in
    \novexp, \lbrace i \in I : d_i \lt C \rbrace \text{ is finite for all } C
    \in \mathbb{R} \biggr\rbrace.
  \]
  For $M$ an $A$-module, we similarly define the $A\dparen[t]$-module
  \[
    M\dparen[t] = \biggl\lbrace \sum_{i \in I} m_i t^{d_i} : m_i \in M, d_i \in
    \novexp, \lbrace i \in I : d_i \lt C \rbrace \text{ is finite for all } C
    \in \mathbb{R} \biggr\rbrace.
  \]
\end{definition}

\begin{remark} \label{Rem:NovikovOneVar}
  We may also describe the ring $A\dparen[t]$ and the module $M\dparen[t]$ as
  follows. We first consider the polynomial ring $A[t^\novexp]$ and the module
  $M[t^\novexp]$ where elements of $A[t^\novexp]$ are finite sums of monomials
  $a_i t^{d_i}$. We endow $A[t^\novexp]$ with the topology where $\bigoplus_{c
  \lt d \in \novexp} A t^d \subseteq A[t^\novexp]$ for $c \in \mathbb{R}$ form a
  neighborhood basis of $0$, and similarly for $M[t^\novexp]$. Then
  $A\dparen[t]$ and $M\dparen[t]$ are the completions of the topological ring
  $A[t^\novexp]$ and the topological module $M[t^\novexp]$.
\end{remark}

\begin{remark}
  Such rings naturally appear as coefficient rings in Lagrangian Floer theory
  \cite{FOOO09}, and are named after Novikov for his use of a formal variable in
  the coefficient ring of a homology theory to keep track of topological
  invariants \cite{Nov81}.
\end{remark}

To formulate our descent problem, we define a version of the Novikov ring with
multiple variables.

\begin{definition}
  Let $\Gamma \subseteq \mathbb{R}$ be an additive submonoid, let $A$ be a
  commutative ring, and let $M$ be an $A$-module. We define the
  \textdef{generalized Novikov ring} with $n$ variables
  \[
    A\dparenmult = \left\lbrace \sum_{i \in I} a_i t_1^{d_{i,1}} \dotsm
    t_n^{d_{i,n}} : \begin{matrix} a_i \in A, d_{i,j} \in \novexp, \lbrace i \in
      I : \sum_j s_j d_{i,j} \lt C \rbrace \text{ is} \\ \text{finite for all }
      C \in \mathbb{R} \text{ and } s_1, \dotsc, s_n \in \mathbb{R}_{\gt 0}
    \end{matrix} \right\rbrace
  \]
  and the $A\dparenmult$-module
  \[
    M\dparenmult = \left\lbrace \sum_{i \in I} m_i t_1^{d_{i,1}} \dotsm
    t_n^{d_{i,n}} : \begin{matrix} m_i \in M, d_{i,j} \in \novexp, \lbrace i \in
      I : \sum_j s_j d_{i,j} \lt C \rbrace \text{ is} \\ \text{finite for all }
      C \in \mathbb{R} \text{ and } s_1, \dotsc, s_n \in \mathbb{R}_{\gt 0}
    \end{matrix} \right\rbrace,
  \]
  where we note that the definitions of $A\dparen[t]$ and $M\dparen[t]$ agree
  with that of Definition~\ref{Def:NovikovOneVar}.
\end{definition}

\begin{example} \label{Exa:Novikov}
  When $\novexp = \mathbb{Z}_{\le 0}$, we have $A\dparen[t][u] = A[t^{-1},
  u^{-1}]$, and when $\novexp = \mathbb{Z}_{\ge 0}$, we have $A\dparen[t][u] =
  A[[t, u]]$. When $\novexp = \mathbb{Z}$,\def\novexp{} the ring
  $A\dparen[t][u]$ is in general much larger than the ring
  $A\dbrack[t][u]{}[(tu)^{-1}]$ since it contains convergent power series such
  as
  \[
    \sum_{n = 0}^\infty t^{-n} u^{n^2} \in A\dparen[t][u].
  \]
  On the other hand, even if $A$ is a field, the ring $A\dparen[t][u]$ is not a
  field because $t-u$ is not invertible.
\end{example}

\begin{remark} \label{Rem:NovikovMultVar}
  \def\multbrack{[t_1^\novexp, \dotsc, t_n^\novexp]}
  \def\multbrackplus{[t_1^{\novexp_{\ge 0}}, \dotsc, t_n^{\novexp_{\ge 0}}]}
  As in Remark~\ref{Rem:NovikovOneVar}, we may construct the Novikov rings
  $A\dparenmult$ as a completion of the polynomial ring $A\multbrack$. Here, we
  given $A\multbrack$ the topology generated by additive cosets of
  \[
    \bigoplus_{d_1, \dotsc, d_n \in \Gamma, \sum_j s_j d_j \ge C} A t_1^{d_1}
    \dotsm t_n^{d_n}
  \]
  where $s_1, \dotsc, s_n \in \mathbb{R}_{\gt 0}$ and $C \in \mathbb{R}$. This
  is indeed a topological ring, and its completion is $A\dparenmult$. One can
  similarly define a topological $A\multbrack$-module $M\multbrack$ whose
  completion is $M\dparenmult$. In particular, it follows from this discussion
  that $A\dparenmult$ indeed has a natural ring structure, and $M\dparenmult$
  indeed has a natural structure of a module over it.
\end{remark}

\begin{remark} \label{Rem:NovikovMultRatLoc}
  In the case when $\novexp \subseteq \mathbb{R}$ is a nontrivial subgroup,
  there is a more geometric interpretation. Fix an element $0 < \gamma \in
  \Gamma$, and consider the ring $A[[t_1^{\Gamma_{\ge 0}}, \dotsc,
  t_n^{\Gamma_{\ge 0}}]]$, which is completion of $A[t_1^{\Gamma_{\ge 0}},
  \dotsc, t_n^{\Gamma_{\ge 0}}]$ with respect to the $(t_1^\gamma, \dotsc,
  t_n^\gamma)$-adic topology. This defines an affinoid pre-adic space
  \[
    X = \Spa(A[[t_1^{\Gamma_{\ge 0}}, \dotsc, t_n^{\Gamma_{\ge 0}}]],
    A[[t_1^{\Gamma_{\ge 0}}, \dotsc, t_n^{\Gamma_{\ge 0}}]]).
  \]
  The locus on which $\lvert t_1 \rvert, \dotsc, \lvert t_n \rvert > 0$ is the
  union of the rational localizations
  \begin{align*}
    \Spa(R_m, R_m^+) &= \lbrace \lvert - \rvert \in X : \lvert t_j \rvert >
    \lvert t_i \rvert^m > 0 \text{ for all } 1 \le i, j \le n \rbrace, \\
    R_m &= \left\lbrace \sum_{i \in I} a_i t_1^{d_{i,1}} \dotsm
    t_n^{d_{i,n}} : \begin{matrix} a_i \in A, d_{i,j} \in \novexp, \lbrace i \in
      I : \sum_j s_j d_{i,j} \lt C \rbrace \text{ is} \\ \text{finite for all }
      C \in \mathbb{R} \text{ and } s_1, \dotsc, s_n \in \mathbb{R}_{\gt 0}
      \\ \text{ with } m \min\lbrace s_i \rbrace \ge \max\lbrace s_i \rbrace
    \end{matrix} \right\rbrace.
  \end{align*}
  Then we have $A\dparenmult = \varprojlim_m R_m$.
\end{remark}

\begin{lemma} \label{Lem:NovikovExact}
  The functor from $A$-modules to $A\dparenmult$-modules sending $M \mapsto
  M\dparenmult$ is exact.
\end{lemma}

\begin{proof}
  Given surjective $A$-module homomorphism $f \colon M_1 \twoheadrightarrow
  M_2$, the induced map
  \[
    f\dparenmult \colon M_1\dparenmult \to M_2\dparenmult
  \]
  surjective because we can lift each nonzero coefficient individually. The
  kernel $\ker f\dparenmult$ consists of those series whose coefficients all lie
  in $\ker f$, and hence is $(\ker f)\dparenmult$.
\end{proof}

In particular, when $M$ is a finitely presented $A$-module, the natural map
\[
  M \otimes_A A\dparenmult \to M\dparenmult
\]
is an isomorphism. As a consequence, for any other $A$-module $N$ there is a
canonical isomorphism
\[
  \Hom_{A\dparenmult}(M\dparenmult, N\dparenmult) \cong \Hom_A(M, N)\dparenmult
\]
of $A\dparenmult$-modules.

\subsection{Novikov descent data}

We define $A$-algebra homomorphisms
\begin{align*}
  \pi_1^\ast &\colon A\dparen[t] \to A\dparen[t][u]; & t &\mapsto t, \\
  \pi_2^\ast &\colon A\dparen[t] \to A\dparen[t][u]; & t &\mapsto u, \\
  \pi_{12}^\ast &\colon A\dparen[t][u] \to
  A\dparen[t][u][v]; & t &\mapsto t, u \mapsto
  u, \\
  \pi_{13}^\ast &\colon A\dparen[t][u] \to
  A\dparen[t][u][v]; & t &\mapsto t, u \mapsto
  v, \\
  \pi_{23}^\ast &\colon A\dparen[t][u] \to
  A\dparen[t][u][v]; & t &\mapsto u, u \mapsto
  v.
\end{align*}
They form an augmented semi-cosimplicial diagram of $A$-algebras
\[
  A \to A\dparen[t] \rightrightarrows A\dparen[t][u] \rrrarrows
  A\dparen[t][u][v].
\]
We will also use the notation $\rho_1^\ast = \pi_{12}^\ast \circ \pi_1^\ast =
\pi_{13}^\ast \circ \pi_1^\ast$ and similarly $\rho_2^\ast = \pi_{12}^\ast \circ
\pi_2^\ast = \pi_{23}^\ast \circ \pi_1^\ast$ and $\rho_3^\ast = \pi_{13}^\ast
\circ \pi_2^\ast = \pi_{23}^\ast \circ \pi_2^\ast$. For a ring homomorphism
$f^\ast \colon A \to B$ and an $A$-module $M$, we will also denote $f^\ast M
= M \otimes_{A, f^\ast} B$, and write $f^\ast \colon M \to f^\ast M$ for the map
$m \mapsto m \otimes 1$ by making an abuse of notation.

\begin{definition} \label{Def:DescentDatum}
  A \textdef{Novikov descent datum} for a ring $A$ is a pair $(M, \varphi)$
  where $M$ is a finite projective module over $A\dparen[t]$ and
  \[
    \varphi \colon \pi_1^\ast M \xrightarrow{\cong} \pi_2^\ast M
  \]
  is an $A\dparen[t][u]$-linear automorphism satisfying the cocycle condition,
  which states that the diagram
  \[ \begin{tikzcd}
    \rho_1^\ast M \arrow{rr}{\pi_{13}^\ast \varphi} \arrow{dr}[']{\pi_{12}^\ast
    \varphi} & & \rho_3^\ast M \\ & \rho_2^\ast M. \arrow{ru}[']{\pi_{23}^\ast
    \varphi}
  \end{tikzcd} \]
  commutes. We denote by $\novdd(A)$ the category of Novikov descent data.
\end{definition}

Denote by $\vect(A)$ the category of finite projective $A$-modules. There is a
natural functor
\[
  \vect(A) \to \novdd(A); \quad M_0 \mapsto (M_0 \otimes_A A\dparen[t], \id_{M_0
  \otimes_A A\dparen[t][u]}).
\]

\begin{remark} \label{Rem:Reformulation}
  Our main theorem, Theorem~\ref{Thm:Main}, may now be reformulated to state
  that the functor $\vect(A) \to \novdd(A)$ is an equivalence of categories.
\end{remark}

\begin{definition} \label{Def:Support}
  For an element $x \in A\dparenmult$ expressed as
  \[
    x = \sum_{i \in I} a_i t_1^{d_{i,1}} \dotsm t_n^{d_{i,n}}
  \]
  where $a_i \neq 0$ for all $i \in I$ and $(d_{i,1}, \dotsc, d_{i,n}) \neq
  (d_{j,1}, \dotsc, d_{j,n})$ for $i \neq j$ in $I$, we define the
  \textdef{support} of $x$ as
  \[
    \supp(x) = \lbrace (d_{i,1}, \dotsc, d_{i,n}) : i \in I \rbrace \subseteq
    \Gamma^n.
  \]
  We similarly define the support of an element of $M\dparenmult$ for $M$ an
  $A$-module.
\end{definition}

The support of $0 \in A\dparenmult$ is empty. We always have $\supp(x + y)
\subseteq \supp(x) \cup \supp(y)$ and $\supp(xy) \subseteq \supp(x) + \supp(y)$,
where $+$ is the Minkowski sum defined as $X + Y = \lbrace x + y : x \in X,
y \in Y \rbrace$. By definition, for every element $x \in A\dparenmult$ or $x
\in M\dparenmult$, the intersection of $\supp(x)$ with the half-space
\[
  \lbrace (d_1, \dotsc, d_n) : s_1 d_1 + \dotsb + s_n d_n < C \rbrace \subseteq
  \mathbb{R}^n
\]
is finite whenever $s_1, \dotsc, s_n > 0$.

\begin{proposition} \label{Prop:NovDescentFullyFaithful}
  The functor $\vect(A) \to \novdd(A)$ is fully faithful.
\end{proposition}

\begin{proof}
  Given finite projective $A$-modules $M_0$ and $N_0$, let
  \[
    f \colon M_0\dparen[t] \to N_0\dparen[t]
  \]
  be an $A\dparen[t]$-linear map compatible with the descent data on both sides.
  This means that the diagram
  \[ \begin{tikzcd}
    M_0\dparen[t][u] \arrow{r}{\pi_1^\ast f} \arrow[equals]{d} &
    N_0\dparen[t][u] \arrow[equals]{d} \\ M_0\dparen[t][u] \arrow{r}{\pi_2^\ast
    f} & N_0 \dparen[t][u]
  \end{tikzcd} \]
  commutes, where both $\pi_1^\ast f$ and $\pi_2^\ast f$ are
  $A\dparen[t][u]$-linear. Viewing $f$ as an element of $\Hom_A(M_0,
  N_0)\dparen[t]$, we see that compatibility with descent data translates to
  \[
    \pi_1^\ast f = \pi_2^\ast f \in \Hom_A(M_0, N_0)\dparen[t][u].
  \]
  We see that $\supp(\pi_1^\ast f) = \supp(f) \times \lbrace 0 \rbrace$ while
  $\supp(\pi_2^\ast f) = \lbrace 0 \rbrace \times \supp(f)$. This immediately
  shows that $\supp(f) \subseteq \lbrace 0 \rbrace$ and hence $f \in \Hom_A(M_0,
  N_0)$.
\end{proof}

\subsection{Topological rings and modules}
\def\novexp{\mathbb{R}}

We start by reviewing general facts about finite projective modules over
topological rings, which we do not assume are Hausdorff. (In practice, we will
apply the results to $A\dparen[t]$, which is always Hausdorff.) A
\textdef{topological ring} is a ring $R$ together with a topology such that the
addition and multiplication maps $R \times R \to R$ are continuous, where $R
\times R$ is given the product topology. A \textdef{topological module} over a
topological ring $R$ is an $R$-module $M$ together with a topology such that the
addition map $M \times M \to M$ and the scalar multiplication map $R \times M
\to M$ are continuous.

\begin{definition}
  Let $R$ be a topological ring, and let $M$ be a finite projective $R$-module.
  The \textdef{canonical topology} on $M$ is defined as the coarsest topology
  under which all $R$-linear homomorphisms $M \to R$ are continuous.
\end{definition}

\begin{lemma} \label{Lem:CanonicalTopology}
  Let $R$ be a topological ring, and let $M, N$ be finite projective
  $R$-modules.
  \begin{enumerate}
    \item The canonical topology on $R$ as an $R$-module agrees with the
      topology on $R$.
    \item The canonical topology on $M$ makes $M$ into a topological $R$-module.
    \item The canonical topology on $M \oplus N$ is the product topology of the
      canonical topologies on $M$ and $N$.
    \item Every $R$-linear map $M \to N$ is continuous with respect to the
      canonical topologies on $M$ and $N$.
  \end{enumerate}
\end{lemma}

\begin{proof}
  The statement (4) follows from unraveling the definition of the coarsest
  topology; the canonical topology on $N$ is generated by preimages of open
  sets $U \subseteq R$ along $R$-linear maps $N \to R$, and its preimage in $M$
  is the preimage of $U$ under the composition $M \to N \to R$, hence open. For
  (1), we note that the canonical topology is at least as fine as the ring
  topology, since $\id \colon R \to R$ is linear. On the other hand, it is fine
  enough because every scaling $r \colon R \to R$ for $r \in R$ is continuous.
  For (3), we similarly note that the canonical topology is at least as fine as
  a product topology, because the projection maps $M \oplus N \to M$ and $M
  \oplus N \to N$ are continuous. It is fine enough, because any $M \oplus N \to
  R$ can be written as a composition
  \[
    M \times N \xrightarrow{f \times g} R \times R \xrightarrow{+} R
  \]
  where $f \colon M \to R$ and $g \colon N \to R$ are $R$-linear, hence
  continuous. Finally, for (2), we choose a splitting $M \oplus M^\prime \cong
  R^{\oplus n}$ of $R$-modules. By (4), both the inclusion $M \hookrightarrow
  R^{\oplus n}$ and the projection $R^{\oplus n} \twoheadrightarrow M$ are
  continuous. The addition and scalar multiplication maps on $M$ can be
  written as the composition
  \[
    M \times M \hookrightarrow R^{\oplus n} \times R^{\oplus n} \xrightarrow{+}
    R^{\oplus n} \twoheadrightarrow M, \quad R \times M \hookrightarrow R \times
    R^{\oplus n} \xrightarrow{\times} R^{\oplus n} \twoheadrightarrow M
  \]
  of continuous maps.
\end{proof}

We focus on the case when there is one variable and the exponent group is
$\Gamma = \mathbb{R}$. For convenience, for $A$ a commutative ring and $M$ an
$A$-module, we denote
\[
  A\dbrack[t] = A\ldparen t^{\mathbb{R}_{\ge 0}} \rdparen, \quad M\dbrack[t] =
  M\ldparen t^{\mathbb{R}_{\ge 0}} \rdparen.
\]

\begin{definition}
  We define the \textdef{$t$-adic topology} on $A\dparen[t]$ to be the topology
  generated by $x + t^d A\dbrack[t]$ for varying $x \in A\dparen[t]$ and $d \in
  \mathbb{R}$. We similarly define the \textdef{$t$-adic topology} on
  $M\dparen[t]$ to be the topology generated by $x + t^d M\dbrack[t]$ for
  varying $x \in M\dparen[t]$ and $d \in \mathbb{R}$.
\end{definition}

It can be readily verified that $A\dparen[t]$ is a topological ring. Indeed,
it is a complete Tate ring in the sense of \cite{Hub93}, where $A\dbrack[t]$ is
a ring of definition. It can also be checked that $M\dparen[t]$ is a complete
topological $A\dparen[t]$-module.

If $M$ is a finite projective $A$-module, then $M\dparen[t]$ is a finite
projective $A\dparen[t]$-module, and moreover the $t$-adic topology on
$M\dparen[t]$ agrees with the canonical topology. We now a prove a converse
statement.

\begin{lemma} \label{Lem:FiniteProjTopology}
  Let $A$ be a commutative ring, and let $M$ be an $A$-module. Assume that
  $M\dparen[t]$ is a finite projective $A\dparen[t]$-module, and moreover the
  $t$-adic topology on $M\dparen[t]$ agrees with the canonical topology. Then
  $M$ is a finite projective $A$-module.
\end{lemma}

\begin{proof}
  We first show that $M$ is finitely generated as an $A$-module. We choose a
  finite collection of generators $x_1, \dotsc, x_n \in M\dparen[t]$ as an
  $A\dparen[t]$-module. Since $M\dparen[t]$ is projective, the natural
  map
  \[
    A\dparen[t]^{\oplus n} \xrightarrow{(x_1, \dotsc, x_n)} M\dparen[t]
  \]
  splits and hence is open by Lemma~\ref{Lem:CanonicalTopology}. This implies
  that there exists a large enough $N$ for which
  \[
    t^N M\dbrack[t] \subseteq \sum_{i=1}^n x_i A\dbrack[t].
  \]
  In particular, for every $m \in M$, the element $t^N m \in M\dbrack[t]$ can be
  written as $A\dbrack[t]$-linear combination of $x_1, \dotsc, x_n$. Therefore
  once we write $x_i = \sum_j m_{ij} t^{d_{ij}}$, the collection of $m_{ij}$
  with $d_{ij} \le N$ generate $M$, and so $M$ is finitely generated.

  We now choose an $A$-linear surjection
  \[
    0 \to K \to A^{\oplus m} \to M \to 0.
  \]
  It follows that
  \[
    0 \to K\dparen[t] \to A\dparen[t]^{\oplus m} \to M\dparen[t] \to 0
  \]
  is short exact, and hence split. This implies that $K\dparen[t]$ is finite
  projective over $A\dparen[t]$. Moreover the canonical topology of
  $K\dparen[t]$ is the subspace topology induced from $A\dparen[t]^{\oplus m}$,
  again by Lemma~\ref{Lem:CanonicalTopology}, and hence agrees with the $t$-adic
  topology. By applying the same argument, we conclude that $K$ is also finitely
  generated, and therefore $M$ is finitely presented.

  This shows that the natural map
  \[
    M \otimes_A A\dparen[t] \to M\dparen[t]
  \]
  is an isomorphism. Since $A \hookrightarrow A\dparen[t]$ splits as a map of
  $A$-modules, it is universally injective. Then
  \cite[Theorem~08XD]{Sta24} applies to show that $M$ is a finite
  projective $A$-module.
\end{proof}

\section{Novikov isocrystals} \label{Sec:Isocrystal}

We fix a real number $\lambda \gt 1$. For each commutative ring $A$, we define
the $A$-algebra automorphism
\[
  F \colon A\dparen[t] \to A\dparen[t]; \quad \sum_{i \in I} a_i t^{d_i} \mapsto
  \sum_{i \in I} a_i t^{\lambda d_i}.
\]

\begin{definition}
  A \textdef{Novikov isocrystal} over a ring $A$ is a pair $(M, F_M)$, where $M$
  is a finite projective $A\dparen[t]$-module and $F_M \colon M \to M$ is a
  $F$-semilinear bijection. A morphism $f \colon (M, F_M) \to (N, F_N)$ between
  two Novikov isocrystals is an $A\dparen[t]$-linear map $f \colon M \to N$
  satisfying $f \circ F_M = F_N \circ f$. We denote by $\novisoc(A)$ the
  category of Novikov isocrystals over $A$.
\end{definition}

\begin{remark}
  The name ``Novikov isocrystal'' comes from the analogy with the notion of a
  F-isocrystal \cite{Man63}. For $k$ a perfect field in characteristic $p > 0$,
  a \textdef{F-isocrystal} over $k$ is a $W(k)[p^{-1}]$-vector space $M$
  together with a $F$-semilinear bijection $F_M \colon M \to M$, where $W(k)$ is
  the ring of $p$-typical Witt vectors on $k$ and $F \colon W(k)[p^{-1}] \to
  W(k)[p^{-1}]$ is the canonical lift of the absolute Frobenius.
\end{remark}

Note that a ring homomorphism $A \to B$ induces a $F$-equivariant ring
homomorphism $A\dparen[t] \to B\dparen[t]$, and hence base changing along this
map defines a functor $\novisoc(A) \to \novisoc(B)$. There is also a natural
functor
\[
  \vect(A) \to \novisoc(A); \quad M_0 \mapsto (M_0 \otimes_A A\dparen[t],
  \id_{M_0} \otimes F).
\]

\begin{proposition} \label{Prop:VectToIsocFullyFaithful}
  For every ring $A$, the functor $\vect(A) \to \novisoc(A)$ is fully
  faithful.
\end{proposition}

\begin{proof}
  Let $M_0$ and $N_0$ be two finite projective $A$-modules, and let $f \colon
  M_0 \otimes_A A\dparen[t] \to N_0 \otimes_A A\dparen[t]$ be an
  $A\dparen[t]$-linear map satisfying $f \circ (\id_{M_0} \otimes F) =
  (\id_{N_0} \otimes F) \circ f$. Regarding $f$ as an
  element of
  \[
    \Hom_A(M_0, N_0) \otimes_A A\dparen[t] = \Hom_A(M_0, N_0)\dparen[t],
  \]
  we see that $f$ must be fixed under the automorphism $\id_{\Hom_A(M_0, N_0)}
  \otimes F$, and hence $\supp(f) = \lambda \supp(f)$. On the other hand, the
  intersection $[-1, 1] \cap \supp(f)$ is finite, and therefore $\supp(f)
  \subseteq \lbrace 0 \rbrace$. This shows that $f \in \Hom_A(M_0, N_0)$.
\end{proof}

\begin{remark}
  \def\vectaut{\mathsf{VectAut}}
  The functor $\vect(A) \to \novisoc(A)$ is far from being essentially
  surjective. If we denote by $\vectaut(A)$ the category of pairs $(M_0,
  F_{M_0})$ where $M_0$ is a finite projective $A$-module and $F_{M_0} \colon
  M_0 \to M_0$ is an $A$-linear automorphism, there is a fully faithful functor
  $\vectaut(A) \to \novisoc(A)$ sending $(M_0, F_{M_0})$ to $(M_0\dparen[t],
  F_{M_0} \otimes F)$. Even then, the functor $\vectaut(A) \to \novisoc(A)$ is
  not necessarily essentially surjective. Already for $k$ a field, the Novikov
  isocrystal
  \[
    M = k\dparen[t]^{\oplus 2}, \quad F_M = \begin{pmatrix} 1 &
    t^{-1} \\ 0 & 1 \end{pmatrix} \circ (F \oplus F)
  \]
  is not contained in the essential image of $\vectaut(k)$.
\end{remark}

\subsection{Some results on isocrystals}

In this subsection, we provide some results that allow us to control the
structure of a Novikov isocrystal. These results will later be used in the proof
of Theorem~\ref{Thm:Main}.

\begin{proposition} \label{Prop:NovIsocGeomPt}
  Let $k$ be an algebraically closed field. Then objects of the form
  $(k\dparen[t], c F)$ for some $c \in k^\times$ generate $\novisoc(k)$ under
  extensions.
\end{proposition}

\begin{proof}
  We imitate the argument in the proof of the Dieudonn\'{e}--Manin
  classification \cite{Man63}. Because $k\dparen[t]$ is a field, it suffices
  to prove that for every nonzero Novikov isocrystal $(M, F_M)$, there exists a
  nonzero element $m \in M$ and $c \in k^\times$ satisfying $F_M(m) = cm$. Fix
  an arbitrary nonzero $m \in M$ and find $n \ge 1$ such that $m, F_M(m),
  \dotsc, F_M^{n-1}(m)$ are linearly independent over $k\dparen[t]$ while
  \[
    F_M^n(m) = a_0 m + a_1 F_M(m) + \dotsb + a_{n-1} F_M^{n-1}(m)
  \]
  for $a_0, \dotsc, a_{n-1} \in k\dparen[t]$. Replacing $m$ with $t^r m$ for a
  real number $r \in \mathbb{R}$ has the effect of replacing $a_i$ with $t^{r
  (\lambda^n - \lambda^i)}$. Thus by choosing $r$ appropriately, we may further
  assume that $a_1, \dotsc, a_n \in k\dbrack[t]$ with $a_i \in
  k\dbrack[t]^\times$ for some $0 \le i \le n-1$.

  Our goal is to find elements $b_0, \dotsc, b_{n-1} \in
  k\dparen[t]$, not all zero, and a scalar $c \in k^\times$
  satisfying the equation
  \[
    F_M(b_0 m + b_1 F_M(m) + \dotsb + b_{n-1} F_M^{n-1}(m))= cb_0 m + cb_1
    F_M(m) + \dotsb + cb_{n-1} F_M^{n-1}(m).
  \]
  This amounts to the system of equations
  \[
    F(b_{n-1}) a_0 = cb_0, \; F(b_0) + F(b_{n-1}) a_1 = cb_1, \; \dotsc, \;
    F(b_{n-2}) + F(b_{n-1}) a_{n-1} = cb_{n-1},
  \]
  which can be rewritten as a single equation
  \[
    F^n(b) F^{n-1}(a_0) + c F^{n-1}(b) F^{n-2}(a_1) + \dotsb + c^{n-1} F(b)
    a_{n-1} = c^n b
  \]
  in terms of $b = b_{n-1}$.

  We claim that this equation has a nonzero solution in $k\dbrack[t]$. For each
  $0 \le i \le n-1$, let $\bar{a}_i \in k$ be the constant term of $a_i \in
  k\dbrack[t]$. Since
  we are assuming that $\bar{a}_i \neq 0$ for some $1 \le i \le n-1$ and that
  $k$ is algebraically closed, there exists a $c \in k^\times$ for which
  \[
    \bar{a}_0 + c \bar{a}_1 + \dotsb + c^{n-1} \bar{a}_{n-1} = c^n.
  \]
  Write
  \[
    g(b) = c^{-n} F^n(b) F^{n-1}(a_0) + c^{-n+1} F^{n-1}(b) F^{n-2}(a_1) +
    \dotsb + c^{-1} F(b) a_{n-1}
  \]
  so that we need to find $b \in k\dbrack[t]$ with $g(b) = b$. Set
  $b^{(0)} = 1$ and $b^{(j+1)} = g(b^{(j)})$. By construction of $c$, the
  constant term of $b^{(1)} = g(b^{(0)})$ is $1$, and hence there exists a
  positive real number $\epsilon > 0$ for which $b^{(1)} - b^{(0)} \in t^\epsilon
  k\dbrack[t]$. On the other hand, we observe that $x - y \in t^d
  k\dbrack[t]$ implies $g(x) - g(y) \in t^{\lambda d} k\dbrack[t]$.
  This shows that $b^{(j+1)} - b^{(j)} \in t^{\lambda^j \epsilon} k\dbrack[t]$
  and therefore
  \[
    b = \lim_{i \to \infty} b^{(j)} \in 1 + t^\epsilon k\dbrack[t]
  \]
  exists and satisfies $g(b) = b$.
\end{proof}

\begin{lemma} \label{Lem:FrobeniusLimit}
  Let $A$ be a commutative ring, and let $M_0$ be an $A$-module. Consider the
  topological $A\dparen[t]$-module $M_0\dparen[t]$, where $M_0\dparen[t]$ has
  the $t$-adic topology. If $S \subseteq M_0\dparen[t]$ is a closed
  $A$-submodule that is stable under the map
  \[
    F \colon M_0\dparen[t] \to M_0\dparen[t]; \quad \sum_i m_i t^{d_i} \mapsto
    \sum_i m_i t^{\lambda d_i},
  \]
  and also stable under multiplication by $t^d$ for every $d \in \mathbb{R}$,
  then there exists an $A$-submodule $S_0 \subseteq M_0$ for which $S =
  S_0\dparen[t]$.
\end{lemma}

\begin{proof}
  Choose any element $s \in S$ and write
  \[
    s = \sum_{i \ge 1} m_i t^{d_i} \in M_0\dparen[t], \quad m_i \in M_0, \quad
    d_1 < d_2 < \dotsb.
  \]
  Since $S$ is closed and stable under $F$, we have
  \[
    m_1 = \lim_{k \to \infty} F^k(t^{-d_1} s) \in S.
  \]
  Then $\sum_{i \ge 2} m_i t^{d_i} = s - m_1 t^{d_1} \in S$ and similarly $m_2
  \in S$. Iterating the process shows that $m_i \in S$ for all $i$, and
  therefore all coefficients that appear in the power series expansion of
  elements of $S$ are in $S$. Let $S_0 \subseteq M_0$ be the subset of all such
  coefficients that appear in elements of $S$, so that $S \subseteq
  S_0\dparen[t]$. The above analysis shows that $S_0 \subseteq S$,
  and by writing each element of $S_0\dparen[t]$ as a convergent sum,
  we moreover obtain $S = S_0\dparen[t]$.
\end{proof}

\begin{proposition} \label{Prop:NovIsocInjective}
  Let $A \hookrightarrow B$ be an injective ring homomorphism, and let $(M,
  F_M) \in \novisoc(A)$ be a Novikov isocrystal with the property that its
  base change to $B$ lies in the essential image of $\vect(B) \to \novisoc(B)$.
  Then $(M, F_M)$ lies in the essential image of $\vect(A) \to \novisoc(A)$.
\end{proposition}

\begin{proof}
  By assumption, there exists a finite projective $B$-module $N_0$ together with
  a $B\dparen[t]$-linear isomorphism
  \[
    M \otimes_{A\dparen[t]} B\dparen[t] \cong N_0
    \otimes_B B\dparen[t]
  \]
  such that the action of $F_M \otimes F$ on the left hand side agrees with the
  action of $\id_{N_0} \otimes F$ on the right hand side.

  Since $M$ is flat over $A\dparen[t]$ and $A\dparen[t]
  \to B\dparen[t]$ is injective, we
  may use the composition
  \[
    M \hookrightarrow M \otimes_{A\dparen[t]} B\dparen[t]
    \cong N_0 \otimes_B B\dparen[t] = N_0\dparen[t]
  \]
  to regard $M$ as an $A\dparen[t]$-submodule of $N_0\dparen[t]$ that is stable
  under the action of $\id_{N_0} \otimes F$. Note that $A\dparen[t]
  \hookrightarrow B\dparen[t]$ is a closed subring with the subspace topology.
  By writing $M$ as a direct summand of a finite free module, we see that $M$ is
  also closed in $N_0\dparen[t]$, and moreover the canonical topology on $M$
  agrees with the subspace topology from $N_0\dparen[t]$.

  Using Lemma~\ref{Lem:FrobeniusLimit}, we see that there exists an
  $A$-submodule $M_0 \subseteq N_0$ for which $M = M_0\dparen[t]$. On the other
  hand, we observe from above that the canonical topology on $M$ agrees with the
  $t$-adic topology. It therefore follows from
  Lemma~\ref{Lem:FiniteProjTopology} that $M_0$ is a finite projective
  $A$-module.
\end{proof}

\begin{proposition} \label{Prop:LatticeIsocrystal}
  Let $A$ be a commutative ring, and let $(M, F_M) \in \novisoc(A)$ be a Novikov
  isocrystal. Consider a finitely generated $A\dbrack[t]$-submodule $M^\circ
  \subseteq M$ with the property that $M^\circ[t^{-1}] = M$. Assume that for
  each $m^\circ \in M^\circ$ there exists a real number $\epsilon > 0$ for which
  \[
    F_M(m^\circ) - m^\circ \in t^\epsilon M^\circ.
  \]
  Then $(M, F_M)$ is in the essential image of $\vect(A) \to \novisoc(A)$.
\end{proposition}

\begin{proof}
  Consider the canonical topology on $M$. We first claim that $t^d M^\circ$ for
  $d \in \mathbb{R}$ form a neighborhood basis of $0 \in M$. To see this, we
  choose the finite set of generators of $M^\circ$ so that
  \[
    f \colon A\dparen[t]^{\oplus n} \twoheadrightarrow M
  \]
  is surjective with the image of $A\dbrack[t]^{\oplus n}$ being $M^\circ$.
  Since $M$ is projective, it split and hence $f$ is open by
  Lemma~\ref{Lem:CanonicalTopology}. This immediately implies that
  $M^\circ$ is open, and hence $t^d M^\circ$ are also open. On the other hand,
  for every open neighborhood $0 \in U$ in $M$, its preimage $f^{-1}(U)
  \subseteq A\dparen[t]^{\oplus n}$ is open and hence contains some $t^d
  A\dbrack[t]^{\oplus n}$. It follows that $t^d M^\circ$ is contained in $U$.

  Note that there is an ideal $(t^{>0}) \subseteq A\dbrack[t]$ generated by the
  elements $t^\epsilon$ for all real numbers $\epsilon \gt 0$. The quotient ring
  $A\dbrack[t] / (t^{>0})$ is naturally identified with $A$. Consider the
  $A$-submodule
  \[
    M_0 = \lbrace m^\circ \in M^\circ : F_M(m^\circ) = m^\circ \rbrace,
  \]
  of $M^\circ$. We now show that the composition
  \[
    M_0 \hookrightarrow M^\circ \twoheadrightarrow M^\circ / (t^{>0}) M^\circ
  \]
  is an isomorphism of $A$-modules. For $m^\circ, n^\circ \in M^\circ$, if
  $m^\circ - n^\circ \in t^d M^\circ$ then $F_M(m^\circ) - F_M(n^\circ) \in
  t^{\lambda d} M^\circ$, because $M^\circ$ is stable under $F_M$. This implies
  that for every $m^\circ \in M^\circ$, the limit
  \[
    \tilde{m}^\circ = \lim_{k \to \infty} F_M^k(m^\circ)
  \]
  exists, because if $F_M(m^\circ) - m^\circ \in t^\epsilon M^\circ$ then
  $F_M^{k+1}(m^\circ) - F_M^k(m^\circ) \in t^{\lambda^k \epsilon} M^\circ$.
  Since $\tilde{m}^\circ \in M_0$ and $\tilde{m}^\circ - m^\circ \in (t^{>0})
  M^\circ$, we conclude that the composition $M_0 \to M^\circ / (t^{>0})
  M^\circ$ is surjective. Similarly, if $n_0, m_0 \in M_0$ satisfy $n_0 - m_0
  \in (t^{>0}) M^\circ$, then by iteratively applying $F_M$ we see that $n_0 -
  m_0 \in t^d M^\circ$ for all $d \in \mathbb{R}$, and hence $n_0 = m_0$.

  The isomorphism $M_0 \cong M^\circ / (t^{>0}) M^\circ$ in particular implies
  that $M_0$ is finitely generated as an $A$-module. Choose a finite set of
  generators $m_1, \dotsc, m_n \in M_0$. Because any element in $1 + (t^{>0})
  \subseteq A\dbrack[t]$ is invertible, Nakayama's lemma implies that $m_1,
  \dotsc, m_n$ generate $M^\circ$ as an $A\dbrack[t]$-module. It follows that
  \[
    A\dparen[t]^{\oplus n} \xrightarrow{(m_1, \dotsc, m_n)} M
  \]
  is surjective. We can further promote it to a surjective map of Novikov
  isocrystals
  \[
    (A\dparen[t]^{\oplus n}, F^{\oplus n}) \xrightarrow{(m_1, \dotsc, m_n)} (M,
    F_M).
  \]

  We now consider the kernel
  \[
    K = \ker(A\dparen[t]^{\oplus n} \to M).
  \]
  This is a closed $A\dparen[t]$-submodule of $(A^{\oplus n})\dparen[t]$ that is
  stable under $F^{\oplus n}$. Lemma~\ref{Lem:FrobeniusLimit} implies that there
  exists an $A$-submodule $K_0 \subseteq A^{\oplus n}$ for which $K =
  K_0\dparen[t]$. On the other hand, $K$ has the subspace topology from
  $A\dparen[t]^{\oplus n}$, and therefore the $t$-adic topology on
  $K_0\dparen[t]$ agrees with the canonical topology. Hence
  Lemma~\ref{Lem:FiniteProjTopology} shows that $K_0$ is a finite projective
  $A$-module. At this point, we have
  \begin{align*}
    (M, F_M) &\cong (A\dparen[t]^{\oplus n}, F^{\oplus n}) /
    (K_0\dparen[t]^{\oplus n}, \id_{K_0} \otimes F) \\ &= ((A^{\oplus
    n}/K_0)\dparen[t], \id_{A^{\oplus n}/K_0} \otimes F).
  \end{align*}
  Finally, as in the proof of Lemma~\ref{Lem:FiniteProjTopology}, $A^{\oplus
  n}/K_0$ being finitely presented over $A$ with $(A^{\oplus n}/K_0)\dparen[t]$
  being finite projective over $A\dparen[t]$ implies that $A^{\oplus n}/K_0$ is
  finite projective over $A$, by \cite[Theorem~08XD]{Sta24}.
\end{proof}

\section{Proof of the main theorem} \label{Sec:Proof}

\subsection{Isocrystal from a descent datum} \label{Subsec:DDToIsoc}

In this subsection, we construct a functor $\novdd(A) \to \novisoc(A)$. First
note that for every finite projective module $M$ over $A\dparen[t]$, the diagram
\[ \begin{tikzcd}
  M \arrow{r}{\pi_2^\ast} & \pi_2^\ast M \arrow[shift left]{r}{\pi_{13}^\ast}
  \arrow[shift right]{r}[']{\pi_{23}^\ast} & \rho_3^\ast M
\end{tikzcd} \]
is an equalizer. Indeed, this can be checked for $M = A\dparen[t]$ by
flatness of $M$, in which case it is clear.

We define the $A$-algebra automorphisms
\begin{align*}
  F_2 &\colon A\dparen[t][u] \to A\dparen[t][u]; & t &\mapsto t, u \mapsto
  u^\lambda, \\ F_3 &\colon A\dparen[t][u][v] \to A\dparen[t][u][v]; & t
  &\mapsto t, u \mapsto u, v \mapsto v^\lambda.
\end{align*}
For each $(M, \varphi) \in \mathsf{NovDD}(A)$, we define a $F_2$-semilinear
bijection
\[
  F_{M,2} \colon \pi_2^\ast M \xrightarrow{\varphi^{-1}} \pi_1^\ast M
  \xrightarrow{\id_M \otimes F_2} \pi_1^\ast M \xrightarrow{\varphi} \pi_2^\ast
  M.
\]
Similarly, we can define a $F_3$-semilinear bijection
\[
  F_{M,3} \colon \rho_3^\ast M \xrightarrow{\pi_{13}^\ast \varphi^{-1}}
  \rho_1^\ast M \xrightarrow{\id_M \otimes F_3} \rho_1^\ast M
  \xrightarrow{\pi_{13}^\ast \varphi} \rho_3^\ast M,
\]
which agrees with the other possible definition
\[
  F_{M,3} \colon \rho_3^\ast M \xrightarrow{\pi_{23}^\ast \varphi^{-1}}
  \rho_2^\ast M \xrightarrow{\id_M \otimes F_3} \rho_2^\ast M
  \xrightarrow{\pi_{23}^\ast \varphi} \rho_3^\ast M,
\]
thanks to the cocycle condition. They fit in a commutative diagram
\[ \begin{tikzcd}
  M \arrow{r}{\pi_2^\ast} \arrow[dashed]{d}{F_M} & \pi_2^\ast M \arrow[shift
  left]{r}{\pi_{13}^\ast} \arrow[shift right]{r}[']{\pi_{23}^\ast}
  \arrow{d}{F_{M,2}} & \rho_3^\ast M \arrow{d}{F_{M,3}} \\ M
  \arrow{r}{\pi_2^\ast} & \pi_2^\ast M \arrow[shift left]{r}{\pi_{13}^\ast}
  \arrow[shift right]{r}[']{\pi_{23}^\ast} & \rho_3^\ast M.
\end{tikzcd} \]
Both rows are equalizers, and all horizontal maps are $A\dparen[t]$-linear while
the vertical maps are $F$-semilinear, if we use the ring homomorphisms
$\pi_2^\ast \colon A\dparen[t] \to A\dparen[t][u]$ and $\rho_3^\ast \colon
A\dparen[t] \to A\dparen[t][u][v]$ to give every module an $A\dparen[t]$-module
structure. Therefore we obtain an $F$-semilinear bijection $F_M \colon M \to M$,
and this defines the functor
\[
  \novdd(A) \to \novisoc(A).
\]

\begin{lemma} \label{Lem:RecoverDDFromIsoc}
  Let $A$ be any ring, and assume that the descent datum $(M, \varphi) \in
  \novdd(A)$ induces $F_M \colon M \to M$. Then the diagram
  \[ \begin{tikzcd}[column sep=large]
    M \arrow{r}{\varphi \circ \pi_1^\ast} \arrow{d}{\varphi \circ \pi_1^\ast} &
    \pi_2^\ast M \arrow[shift left]{r}{\id} \arrow[shift right]{r}[']{F_M
    \otimes F_2} \arrow{d}{\pi_{13}^\ast} & \pi_2^\ast M
    \arrow{d}{\pi_{13}^\ast} \\ \pi_2^\ast M \arrow{r}{\pi_{23}^\ast \varphi
    \circ \pi_{12}^\ast} & \rho_3^\ast M \arrow[shift left]{r}{\id} \arrow[shift
    right]{r}[']{F_M \otimes F_3} & \rho_3^\ast M
  \end{tikzcd} \]
  commutes, where both rows are equalizers.
\end{lemma}

\begin{proof}
  The commutation of the left square is the cocycle condition, and the
  commutation of the two right squares is clear. We now verify that the first
  row is an equalizer. Note that because $F_{M,2} \colon \pi_2^\ast M \to
  \pi_2^\ast M$ is $F_2$-semilinear and its restriction along $\pi_2^\ast \colon
  M \hookrightarrow \pi_2^\ast M$ is $F_M$, we have $F_{M,2} = F_M \otimes F_2$.
  Using the definition of $F_{M,2}$, we see that the first row being an
  equalizer is equivalent to
  \[ \begin{tikzcd}
    M \arrow{r}{\pi_1^\ast} & \pi_1^\ast M \arrow[shift left]{r}{\id}
    \arrow[shift right]{r}[']{\id_M \otimes F_2} & \pi_1^\ast M
  \end{tikzcd} \]
  being an equalizer. Since all maps are $A\dparen[t]$-linear if we
  use $\pi_1^\ast$ to give $\pi_1^\ast M$ the structure of an
  $A\dparen[t]$-module, we may use finite projectivity of $M$ to reduce to
  \[ \begin{tikzcd}
    A\dparen[t] \arrow{r}{\pi_1^\ast} & A\dparen[t][u] \arrow[shift
    left]{r}{\id} \arrow[shift right]{r}[']{F_2} & A\dparen[t][u]
  \end{tikzcd} \]
  being an equalizer, which is now clear. The second row can similarly be
  seen to be an equalizer.
\end{proof}

\begin{proposition} \label{Prop:NovToIsocFullyFaithful}
  For every ring $A$, the functor $\novdd(A) \to \novisoc(A)$ is fully
  faithful.
\end{proposition}

\begin{proof}
  It suffices to show that for two descent data $(M, \varphi_M), (N, \varphi_N)
  \in \novdd(A)$ and an $A\dparen[t]$-linear map $f \colon M \to N$,
  if $f \circ F_M = F_N \circ f$ then $\pi_2^\ast f \circ \varphi_M = \varphi_N
  \circ \pi_1^\ast f$.

  We consider the two diagrams from Lemma~\ref{Lem:RecoverDDFromIsoc}
  corresponding to $(M, \varphi_M)$ and $(N, \varphi_N)$. Since $f \circ F_M =
  F_N \circ f$ and the right squares on the diagram only involve $F_M, F_N$, the
  universal property of equalizers produces unique maps $g \colon M \to N$ and
  $h \colon \pi_2^\ast M \to \pi_2^\ast N$ fitting in the commutative diagram
  \[ \begin{tikzcd}[row sep=small, column sep=large]
    M \arrow{rr}{\varphi_M \circ \pi_1^\ast} \arrow{dd}{\varphi_M \circ
    \pi_1^\ast} \arrow{dr}{g} & & \pi_2^\ast M \arrow{dd}[near
    start]{\pi_{13}^\ast} \arrow{dr}{\pi_2^\ast f} \\ & N \arrow{rr}[near
    start]{\varphi_N \circ \pi_1^\ast} \arrow{dd}[near start]{\varphi_N \circ
    \pi_1^\ast} & & \pi_2^\ast N \arrow{dd}{\pi_{13}^\ast} \\ \pi_2^\ast M
    \arrow{rr}[near start]{\pi_{23}^\ast \varphi_M \circ \pi_{12}^\ast}
    \arrow{dr}[']{h} & & \rho_3^\ast M \arrow{dr}{\rho_3^\ast f} \\ & \pi_2^\ast
    N \arrow{rr}{\pi_{23}^\ast \varphi_N \circ \pi_{12}^\ast} & & \rho_3^\ast N.
  \end{tikzcd} \]
  We note that $g$ must be $A\dparen[t]$-linear, since all maps in
  the top face are $A\dparen[t]$-linear when we use $\pi_1^\ast
  \colon A\dparen[t] \to A\dparen[t][u]$ for
  the $A\dparen[t]$-module structure. Next, from the commutativity of
  the top face we see that the bottom face commutes when we replace $h$ with
  $\pi_2^\ast g$. By uniqueness of the morphisms, this implies that $h =
  \pi_2^\ast g$. On the other hand, by comparing the left face with the top
  face we obtain $h = \pi_2^\ast f$, as both $h$ and $\pi_2^\ast f$ are
  $A\dparen[t][u]$-linear while the image of $\varphi_M \circ \pi_1^\ast$
  generates $\pi_2^\ast M$. Since $M, N$ are both finite
  projective and $\pi_2^\ast \colon A\dparen[t] \to
  A\dparen[t][u]$, is injective, it follows from
  $\pi_2^\ast f = \pi_2^\ast g$ that $f = g$. Unraveling the top face now
  gives $\pi_2^\ast f \circ \varphi_M = \varphi_N \circ \pi_1^\ast f$.
\end{proof}

\subsection{The case of a reduced ring} \label{Subsec:Reduced}

We now prove Theorem~\ref{Thm:Main} when $A$ is reduced and $\Gamma =
\mathbb{R}$. The first step is to prove it when $A$ is an algebraically closed
field.

\begin{proposition} \label{Prop:NovDescentForGeomPts}
  Let $k$ be an algebraically closed field. Then the functor $\vect(k) \to
  \novdd(k)$ is an equivalence of categories.
\end{proposition}

\begin{proof}
  In view of Proposition~\ref{Prop:NovDescentFullyFaithful}, it suffices to show
  that the functor is essentially surjective. Then by
  Proposition~\ref{Prop:NovToIsocFullyFaithful}, it is enough to check that the
  image of $\novdd(k) \to \novisoc(k)$ is contained in the essential image of
  $\vect(k) \to \novisoc(k)$.

  Assume that $(M, \varphi) \in \novdd(k)$ induces a Novikov isocrystal $(M,
  F_M) \in \novisoc(k)$. By Proposition~\ref{Prop:NovIsocGeomPt}, there exists a
  filtration
  \[
    0 = M_0 \hookrightarrow M_1 \hookrightarrow \dotsb \hookrightarrow M_n = M
  \]
  with the property that each associated graded $M_i / M_{i-1}$ is isomorphic to
  $(k\dparen[t], c_i F)$ for some $c_i \in k^\times$.

  We inductively prove that each $(M_i, F_M \vert_{M_i}) \in \novisoc(k)$ is in
  the essential image of $\vect(k) \to \novisoc(k)$, where the base case $i = 0$
  is clear. Assume we have an isomorphism
  \[
    (M_{i-1}, F_M \vert_{M_{i-1}}) \cong (k\dparen[t], F)^{\oplus
    (i-1)}.
  \]
  Fixing a splitting of the short exact sequence $0 \to M_{i-1} \to M_i \to
  M_i/M_{i-1} \to 0$, just as $k\dparen[t]$-vector spaces, we find an
  isomorphism $M_i \cong k\dparen[t]^{\oplus i}$ under which $F_M \vert_{M_i}$
  corresponds to
  \[
    F_M \vert_{M_i} = \begin{pmatrix} I_{i-1} & C \\ 0 & c_i \end{pmatrix} F,
  \]
  where $I_{i-1}$ denotes the identity matrix and $C = (c_1, \dotsc, c_{i-1})^T
  \in k\dparen[t]^{\oplus (i-1)}$ is a column vector.

  On the other hand, by Lemma~\ref{Lem:RecoverDDFromIsoc}, the map
  \[
    \id - F_M \vert_{M_i} \otimes F_2 \colon \pi_2^\ast M_i \to \pi_2^\ast M_i
  \]
  has an $i$-dimensional kernel. Since the kernel when restricted to $\pi_2^\ast
  M_{i-1}$ is $(i-1)$-dimensional, there exist $x_1, \dotsc, x_i \in
  k\dparen[t][u]$ with $x_i \neq 0$ satisfying
  \[
    x_1 = F_2(x_1) + (\pi_2^\ast c_1) F_2(x_i), \; \dotsc, \; x_{i-1} =
    F_2(x_{i-1}) + (\pi_2^\ast c_{i-1}) F_2(x_i), \; x_i = c_i F_2(x_i).
  \]
  Because $c_i \in k^\times$ and $x_i \neq 0$, it follows from the last equation
  that $x_i \in k\dparen[t]^\times \subseteq k\dparen[t][u]$ and $c_i = 1$. Then
  setting $y_j = x_j / x_i$, we obtain
  \[
    y_1 = F_2(y_1) + \pi_2^\ast c_1, \quad \dotsb, \quad y_{i-1} = F_2(y_{i-1})
    + \pi_2^\ast c_{i-1}.
  \]
  For each $y_j = \sum_{\alpha \in I_j} a_{\alpha,j} t^{d_{\alpha,j}}
  u^{e_{\alpha,j}}$, we define
  \[
    z_j = \sum_{\alpha \in I_j, d_{\alpha,j} = 0} a_{\alpha,j} t^{e_{\alpha,j}}
    \in k\dparen[t].
  \]
  Then we have
  \[
    z_1 = F(z_1) + c_1, \quad \dotsb, \quad z_{i-1} = F(z_{i-1}) + c_{i-1}.
  \]

  Finally, we can write
  \[
    F_M \vert_{M_i} = \begin{pmatrix} I_{i-1} & C \\ 0 & 1 \end{pmatrix} F =
    \begin{pmatrix} I_{i-1} & Z \\ 0 & 1 \end{pmatrix} F \begin{pmatrix}
    I_{i-1} & Z \\ 0 & 1 \end{pmatrix}^{-1},
  \]
  where $Z = (z_1, \dotsc, z_{i-1})^T \in k\dparen[t]^{\oplus(i-1)}$. This shows
  that under this change of basis, we have $(M_i, F_M \vert_{M_i}) \cong
  (k\dparen[t], F)^{\oplus i}$.
\end{proof}

\begin{proposition} \label{Prop:NovDescentForProdGeomPts}
  Let $A = \prod_{i \in I} k_i$ be a product of algebraically closed fields.
  Then the functor $\vect(A) \to \novdd(A)$ is an equivalence.
\end{proposition}

\begin{proof}
  By Proposition~\ref{Prop:NovDescentFullyFaithful}, it remains to check that
  the functor is essentially surjective. Let $(M, \varphi) \in \novdd(A)$ be
  a descent datum. By Proposition~\ref{Prop:NovDescentForGeomPts}, for each $i$
  there exists a $k_i$-vector space $M_{0,i}$ and an isomorphism
  \[
    f_i \colon M \otimes_{A\dparen[t]} k_i\dparen[t] \cong
    M_{0,i} \otimes_{k_i} k_i\dparen[t],
  \]
  compatible with the descent data. Say $M$ is generated by $N$ elements. Then
  $M \otimes_{A\dparen[t]} k_i\dparen[t]$ has dimension at
  most $N$, and hence $\dim_{k_i} M_{0,i} \le N$. This uniform bound implies
  that $M_0 = \prod_{i \in I} M_{0,i}$ is finite projective over $A$. It remains
  to show that $(M, \varphi) \cong (M_0 \otimes_A A\dparen[t], \id)$.

  Taking the product of $f_i$, we obtain an isomorphism
  \[
    f = \prod_{i \in I} f_i \colon M \otimes_{A\dparen[t]} \prod_{i
    \in I} k_i\dparen[t] \cong M_0 \otimes_A \prod_{i \in I}
    k_i\dparen[t].
  \]
  Compatibility of each $f_i$ with descent data means that
  \[
    \pi_2^\ast f_i \circ (\varphi \otimes_{A\dparen[t][u]}
    k_i\dparen[t][u]) = \pi_1^\ast f_i,
  \]
  and taking the product over all $i$ gives a commutative diagram
  \[ \begin{tikzcd}[column sep=large]
    \pi_1^\ast M \otimes_{A\dparen[t][u]} \prod_{i \in I} k_i\dparen[t][u]
    \arrow{r}{(\prod_i \pi_1)^\ast f} \arrow{d}{\varphi \otimes \id} & M_0
    \otimes_A \prod_{i \in I} k_i\dparen[t][u] \arrow[equals]{d} \\ \pi_2^\ast M
    \otimes_{A\dparen[t][u]} \prod_{i \in I} k_i\dparen[t][u] \arrow{r}{(\prod_i
    \pi_2)^\ast f} & M_0 \otimes_A \prod_{i \in I} k_i\dparen[t][u].
  \end{tikzcd} \]
  In particular, if we consider $M$ as an $A\dparen[t]$-submodule of
  $M_0 \otimes_A \prod_{i \in I} k_i\dparen[t]$ via $f$, the images
  of
  \begin{align*}
    \pi_1^\ast M &= M \otimes_{A\dparen[t], \pi_1^\ast} A\dparen[t][u]
    \hookrightarrow M_0 \otimes_A \prod_{i \in I} k_i\dparen[t][u], \\
    \pi_2^\ast M &= M \otimes_{A\dparen[t], \pi_2^\ast} A\dparen[t][u]
    \hookrightarrow M_0 \otimes_A \prod_{i \in I} k_i\dparen[t][u],
  \end{align*}
  agree.

  For every element $\alpha = (\alpha_i) \in M_0 \otimes_A \prod_{i \in I}
  k_i\dparen[t][u]$, define the support of $\alpha$ as $\supp(\alpha) =
  \bigcup_{i \in I} \supp(\alpha_i) \subseteq \mathbb{R}^2$, where the support
  of $\alpha_i \in M_{0,i}\dparen[t][u]$ is as in Definition~\ref{Def:Support}.
  For $\beta = (\beta_i) \in M_0 \otimes_A \prod_{i \in I} k_i\dparen[t]$ we
  similarly define $\supp(\beta) = \bigcup_i \supp(\beta_i) \subseteq
  \mathbb{R}^2$.

  We now claim that for each $m \in M \subseteq M_0 \otimes_A \prod_{i \in I}
  k_i\dparen[t]$, the support $\supp(m) \subseteq \mathbb{R}$ is
  bounded below. Choose $m_1, \dotsc, m_n \in M \subseteq M_0 \otimes_A \prod_{i
  \in I} k_i\dparen[t]$ a finite collection of generators as an
  $A\dparen[t]$-module. We can find elements $a_{jk} \in
  A\dparen[t][u]$ for $1 \le j, k \le n$ satisfying
  \begin{equation}
    \pi_1^\ast m_j = \sum_{k=1}^n a_{jk} \pi_2^\ast m_k \in M_0 \otimes_A
    \prod_{i \in I} k_i\dparen[t][u]. \tag{$\ast$}
  \end{equation}
  Fix a sufficiently small constant $C \in \mathbb{R}$ satisfying the property
  that
  \[
    \bigcup_{j,k=1}^n \supp(a_{jk}) \subseteq \lbrace (x, y) \in \mathbb{R}^2 :
    2x + y \ge C \rbrace.
  \]
  For each $i \in I$ and $1 \le j \le n$, the set $\supp(m_{j,i})$ has a minimal
  element, and hence $D_i = \min \bigcup_{j=1}^n \supp(m_{j,i})$ exists. From
  ($\ast$) we see that
  \[
    \supp(m_{j,i}) \times \lbrace 0 \rbrace \subseteq \bigcup_{k=1}^n
    (\supp(a_{jk,i}) + (\lbrace 0 \rbrace \times \supp(m_{k,i})))
  \]
  inside $\mathbb{R}^2$, where by $+$ we mean the Minkowski sum. By our
  assumptions, the right hand side is always contained in $\lbrace (x, y) : 2x +
  y \ge C + D_i \rbrace$. On the other hand, for some $1 \le j \le n$, the left
  hand side contains $(D_i, 0)$. This shows that $2 D_i \ge C + D_i$, and hence
  $D_i \ge C$. This shows that $\supp(m_{j,i}) \subseteq [C, \infty)$ for all
  $i, j$, and taking the union over all $i$, we obtain $\supp(m_j) \subseteq [C,
  \infty)$ for all $1 \le j \le n$. For an arbitrary $m \in M$, we can now write
  $m = \sum_{j=1}^n a_j m_j$ for $a_j \in A\dparen[t]$, and since
  each $\supp(a_j)$ is bounded below, we conclude that $\supp(m)$ is bounded
  below as well.

  Recall that the images of $\pi_1^\ast M \to M_0 \otimes_A \prod_{i \in I}
  k_i\dparen[t][u]$ and $\pi_2^\ast M \to M_0 \otimes_A \prod_{i \in I}
  k_i\dparen[t][u]$ agree. For each element $\xi$ in this common image, we can
  write $\xi = \sum_{j=1}^n a_j \pi_1^\ast m_j$, where $a_j \in A\dparen[t][u]$
  and $m_j \in M$. Then
  \[
    \supp(\xi) \subseteq \bigcup_j (\supp(a_j) + (\supp(m_j) \times \lbrace 0
    \rbrace))
  \]
  inside $\mathbb{R}^2$. Because each $\supp(m_j)$ is bounded below, for every
  $C \in \mathbb{R}$, the intersection
  \[
    \supp(\xi) \cap \lbrace (x, y) \in \mathbb{R}^2 : x + y \le C \rbrace
  \]
  has only finitely many $y$-coordinate values. For an arbitrary $m \in M$,
  applying this to $\xi = \pi_2^\ast m$ shows that the intersection $\supp(m)
  \cap (-\infty, C]$ is finite. That is, we have
  \[
    M \subseteq M_0 \otimes_A A\dparen[t].
  \]

  It remains to show that this inclusion is an equality. We apply the entire
  argument to the dual descent datum $M^\vee$ and the isomorphisms
  \[
    (f_i^{-1})^\vee \colon M^\vee \otimes_{A\dparen[t]} k_i\dparen[t] \cong
    M_{0,i}^\vee \otimes_{k_i} k_i\dparen[t].
  \]
  Then we similarly obtain that the image of the composition
  \[
    M^\vee \hookrightarrow M^\vee \otimes_{A\dparen[t]} \prod_{i \in
    I} k_i\dparen[t] \xrightarrow{(f^{-1})^\vee} M_0^\vee \otimes_A
    \prod_{i \in I} k_i\dparen[t]
  \]
  lies in $M_0^\vee \otimes_A A\dparen[t]$. Dualizing the resulting map induces
  an $A\dparen[t]$-linear homomorphism $M_0 \otimes_A A\dparen[t] \to M$ whose
  base change to $\prod_{i \in I} k_i\dparen[t]$ recovers $f^{-1}$. Hence the
  composition
  \[
    M_0 \otimes_A A\dparen[t] \to M \hookrightarrow M_0 \otimes_A A\dparen[t]
  \]
  is the identity map upon base changing to $\prod_{i \in I} k_i\dparen[t]$. On
  the other hand, both sides are finite projective and $A\dparen[t]
  \hookrightarrow \prod_{i \in I} k_i\dparen[t]$ is injective, and thus the map
  is the identity before base change. Therefore $M = M_0 \otimes_A A\dparen[t]$.
\end{proof}

\begin{proposition} \label{Prop:NovDescentForReduced}
  Let $A$ be a reduced ring. Then the functor $\vect(A) \to \novdd(A)$ is an
  equivalence.
\end{proposition}

\begin{proof}
  In view of Proposition~\ref{Prop:NovDescentFullyFaithful}, it suffices to show
  that the functor is essentially surjective. Then by
  Proposition~\ref{Prop:NovToIsocFullyFaithful}, it is enough to check that the
  image of $\novdd(A) \to \novisoc(A)$ is contained in the essential image of
  $\vect(A) \to \novisoc(A)$. Consider any $(M, F_M) \in \novisoc(A)$ induced
  from $(M, \varphi_M) \in \novdd(A)$. Since $A$ is reduced, we can embed $A$
  into a product of algebraically closed fields as
  \[
    A \hookrightarrow B = \prod_{\mathfrak{p}}
    \overline{\operatorname{Frac}(A/\mathfrak{p})},
  \]
  where $\mathfrak{p}$ ranges over all prime ideals of $A$. Base changing along
  this ring homomorphism induces the descent datum $(M, \varphi_M)
  \otimes_{A\dparen[t]} B\dparen[t]$. By
  Proposition~\ref{Prop:NovDescentForProdGeomPts}, this is contained in the
  essential image of $\vect(B) \to \novdd(B)$. Hence the associated isocrystal
  $(M, F_M) \otimes_{A\dparen[t]} B\dparen[t]$ lies in the essential image of
  $\vect(B) \to \novisoc(B)$. We can now apply
  Proposition~\ref{Prop:NovIsocInjective} to conclude that $(M, F_M)$ is in the
  essential image of $\vect(A) \to \novisoc(A)$.
\end{proof}

\subsection{The case of a general ring} \label{Subsec:General}

We now prove Theorem~\ref{Thm:Main} when $A$ is an arbitrary commutative
ring and $\Gamma = \mathbb{R}$. We begin by studying the square-zero deformation
theory of Novikov descent data.

\begin{lemma} \label{Lem:VectSquareZeroDeform}
  Let $A$ be a commutative ring, and let $I \subseteq A$ be an ideal satisfying
  $I^2 = 0$.
  \begin{enumerate}
    \item Let $M, N$ be two finite projective $A$-modules together with an
      isomorphism $\bar{f} \colon M/IM \cong N/IN$ of $A/I$-modules. Then
      $\bar{f}$ lifts to an isomorphism $f \colon M \cong N$ of $A$-modules.
    \item Let $\bar{M}$ be a finite projective $A/I$-module. Then there exists a
      finite projective $A$-module $M$ together with an isomorphism $M/IM \cong
      \bar{M}$.
  \end{enumerate}
\end{lemma}

\begin{proof}
  (1) Letting $P = \Hom_A(M, N)$, we observe that $P$ is again a finite
  projective $A$-module with $\bar{f} \in P/IP$. Lift $\bar{f}$ to any element
  $f \in P$. We now check that $f$ is indeed an isomorphism $f \colon M \to N$
  because we can work locally so that both $M$ and $N$ are free, and then the
  determinant being a unit can be checked modulo $I$.

  (2) Note that $\Spec A$ is homeomorphic to $\Spec A/I$. There exists a cover
  $\lbrace U_i \rbrace$ of $\Spec A \cong \Spec A/I$ consisting of affine opens
  over which $\bar{M}$ is free. We lift $\bar{M} \vert_{U_i}$ to $M_i$ on each
  $U_i$ separately, and use (1) to choose isomorphisms $M_i \vert_{U_i \cap U_j}
  \cong M_j \vert_{U_i \cap U_j}$ lifting the identity map on $\bar{M}
  \vert_{U_i \cap U_j}$. These isomorphisms need not satisfy the cocycle
  condition, but the induced $2$-cocycle takes values in
  \[
    \ker(\mathrm{GL}_A(M_i \vert_{U_i}) \to \mathrm{GL}_{A/I}(\bar{M}
    \vert_{U_i})) \cong \Hom_{A/I}(\bar{M}, I \otimes_{A/I} \bar{M})
    \vert_{U_i}.
  \]
  The second \v{C}ech cohomology group of the quasi-coherent sheaf
  $\Hom_{A/I}(\bar{M}, I \otimes_{A/I} \bar{M})$ with respect to $\lbrace
  U_i \rbrace$ vanishes on $\Spec A/I$ as each $U_i$ is affine. Therefore the
  $2$-cocycle is a coboundary, and hence we may modify the isomorphisms $M_i
  \vert_{U_i \cap U_j} \cong M_j \vert_{U_i \cap U_j}$ so that they do satisfy
  the cocycle condition.
\end{proof}

\begin{lemma} \label{Lem:NovDescentSquareZero}
  Let $A$ be a commutative ring, and let $I \subseteq A$ be an ideal satisfying
  $I^2 = 0$. Let $(M, \varphi) \in \novdd(A)$ be a descent datum satisfying the
  property that its base change to $A/I$ lies in the essential image of
  $\vect(A/I) \to \novdd(A/I)$. Then $(M, \varphi)$ is in the essential image of
  $\vect(A) \to \novdd(A)$.
\end{lemma}

\begin{proof}
  By assumption, there exists a finite projective $A/I$-module $\bar{M}_0$
  together with an isomorphism
  \[
    (M/I\dparen[t] M, \bar{\varphi}) \cong (\bar{M}_0\dparen[t], \id)
  \]
  in $\novdd(A/I)$. Using Lemma~\ref{Lem:VectSquareZeroDeform}, we choose a
  finite projective $A$-module $M_0$ with $M_0/IM_0 \cong \bar{M}_0$. As
  $A\dparen[t] \twoheadrightarrow (A/I)\dparen[t]$ is also a square-zero
  thickening, we may use Lemma~\ref{Lem:VectSquareZeroDeform} again to choose an
  isomorphism
  \[
    M \cong M_0\dparen[t]
  \]
  lifting the isomorphism $M/I\dparen[t] M \cong \bar{M}_0\dparen[t] \cong
  M_0\dparen[t] / (IM_0)\dparen[t]$. Under this isomorphism, we obtain a new
  cocycle
  \[
    \psi \colon M_0\dparen[t][u] \xrightarrow{\cong} M_0\dparen[t][u]
  \]
  that reduces to the identity map on $(M_0/IM_0)\dparen[t][u]$.

  It now suffices to show that $\psi$ is a coboundary, i.e., there exists an
  $A\dparen[t]$-linear automorphism $\xi \colon M_0\dparen[t] \cong
  M_0\dparen[t]$ for which $\psi = \pi_2^\ast \xi^{-1} \circ \pi_1^\ast \xi$. We
  first write
  \[
    \psi = \id + \psi_+, \quad \psi \colon M_0\dparen[t][u] \to
    (IM_0)\dparen[t][u].
  \]
  Since $I$ squares to zero, the cocycle condition on $\psi$ translates to an
  additive cocycle condition
  \[
    \pi_{13}^\ast \psi_+ = \pi_{23}^\ast \psi_+ + \pi_{12}^\ast \psi_+,
  \]
  as $A\dparen[t][u][v]$-linear maps $M_0\dparen[t][u][v] \to
  (IM_0)\dparen[t][u][v]$. Here, we note that $(IM_0)\dparen[t][u][v] =
  I\dparen[t][u][v] M_0\dparen[t][u][v]$, and this justifies that $\pi_{23}^\ast
  \psi_+ \circ \pi_{12}^\ast \psi_+ = 0$.

  As $M_0$ is finite projective over $A$, there are canonical isomorphisms
  \begin{align*}
    \Hom_{A\dparen[t][u]}(M_0\dparen[t][u], (IM_0)\dparen[t][u]) &\cong
    \Hom_A(M_0, IM_0)\dparen[t][u], \\
    \Hom_{A\dparen[t][u][v]}(M_0\dparen[t][u][v], (IM_0)\dparen[t][u][v]) &\cong
    \Hom_A(M_0, IM_0)\dparen[t][u][v]
  \end{align*}
  by Lemma~\ref{Lem:NovikovExact}.
  Using this isomorphism, we identify $\psi_+$ as an element of $\Hom_A(M_0,
  IM_0)\dparen[t][u]$, and similarly the additive cocycle condition as an
  equation valued in the group $\Hom_A(M_0, IM_0)\dparen[t][u][v]$.

  We now note that the cocycle condition implies
  \begin{align*}
    \lbrace (x, 0, z) : (x, z) &\in \supp(\psi_+) \rbrace = \supp(\pi_{13}^\ast
    \psi_+) \\ &\subseteq \supp(\pi_{23}^\ast \psi_+) \cup \supp(\pi_{12}^\ast
    \psi_+) \subseteq (\lbrace 0 \rbrace \times \mathbb{R}^2) \cup (\mathbb{R}^2
    \times \lbrace 0 \rbrace).
  \end{align*}
  It follows $\supp(\psi_+)$ is contained in $(\lbrace 0 \rbrace \times
  \mathbb{R}) \cup (\mathbb{R} \times \lbrace 0 \rbrace)$. This allows us to
  write
  \[
    \psi_+ = c + \sum_{0 \neq d \in \mathbb{R}} (a_d t^a + b_d u^d), \quad c,
    a_d, b_d \in \Hom_A(M_0, IM_0).
  \]
  The cocycle condition then can be written out as
  \[
    c + \sum_{0 \neq d \in \mathbb{R}} (a_d t^d + b_d v^d) = 2c + \sum_{0 \neq d
    \in \mathbb{R}} (a_d t^d + (b_d + a_d) u^d + b_d v^d)
  \]
  and hence $c = 0$ and $b_d = -a_d$ for all $0 \neq d \in \mathbb{R}$. Setting
  $\xi_+ = \sum_{0 \neq d \in \mathbb{R}} a_d t^d$, we obtain
  \[
    \psi_+ = \pi_1^\ast \xi_+ - \pi_2^\ast \xi_+, \quad \xi_+ \in \Hom_A(M_0,
    IM_0)\dparen[t].
  \]
  Regarding $\xi_+$ as an $A\dparen[t]$-linear map $M_0\dparen[t] \to
  (IM_0)\dparen[t]$, we now define $\xi = \id + \xi_+$. It follows that
  \[
    \psi = \id + \pi_1^\ast \xi_+ - \pi_2^\ast \xi_+ = (\id - \pi_2^\ast \xi_+)
    \circ (\id + \pi_1^\ast \xi_+) = \pi_2^\ast \xi^{-1} \circ \pi_1^\ast \xi.
  \]
  This finishes the proof.
\end{proof}

\begin{proposition} \label{Prop:NovDescentR}
  Let $A$ be a commutative ring. Then the functor $\vect(A) \to \novdd(A)$ is an
  equivalence.
\end{proposition}

\begin{proof}
  By Proposition~\ref{Prop:NovDescentFullyFaithful}, it suffices to show that
  the functor is essentially surjective. Consider a Novikov descent datum $(M,
  \varphi) \in \novdd(A)$, and let $(M, F_M)$ be the Novikov isocrystal attached
  to it. Let $I \subseteq A$ be the nilradical so that $B = A/I$ is reduced. By
  Proposition~\ref{Prop:NovDescentForReduced}, there exists a finite projective
  $B$-module $N_0$ for which
  \[
    (M \otimes_{A\dparen[t]} B\dparen[t], F_M \otimes F) \cong (N_0
    \otimes_B B\dparen[t], \id_{N_0} \otimes F)
  \]
  as Novikov isocrystals over $B$.

  Choose a finite set of generators $n_1, \dotsc, n_k \in N_0$ as $B$-modules.
  Because $A\dparen[t] \to B\dparen[t]$ is surjective, there is also a
  surjection
  \[
    M \to M \otimes_{A\dparen[t]} B\dparen[t] \cong N_0\dparen[t].
  \]
  Lift the elements $n_1, \dotsc, n_k \in N_0 \subseteq N_0\dparen[t]$ to
  elements $m_1, \dotsc, m_k \in M$. Since elements of $1 + I\dparen[t]$ are
  invertible in $A\dparen[t]$, the elements $m_1, \dotsc, m_k$ generate $M$ as a
  $A\dparen[t]$-module by Nakayama's lemma. This implies that if we set
  \[
    M^\circ = A\dbrack[t] m_1 + \dotsb + A\dbrack[t] m_k
  \]
  then $M^\circ[t^{-1}] = M$.

  For each $1 \le i \le k$, we have
  \[
    F_M(m_i) - m_i \in \ker(M \to M \otimes_{A\dparen[t]} B\dparen[t]) = M
    \otimes_{A\dparen[t]} I\dparen[t].
  \]
  That is, we have
  \[
    F_M(m_i) \in m_i + I\dparen[t] m_1 + I\dparen[t] m_2 + \dotsb + I\dparen[t]
    m_k.
  \]
  Separating out the monomials with non-positive degree, we can further write
  \[
    F_M(m_i) - m_i - \sum_{j \in I_i} a_{i,j} t^{d_{i,j}} m_{l_{i,j}} \in
    \sum_{i=1}^k (t^{>0}) I\dbrack[t] m_i \subseteq (t^{>0}) M^\circ,
  \]
  where $I_i$ is finite and $a_{i,j} \in I$. That is, if we consider the
  finitely generated ideal $J \subseteq I$ generated by $a_{i,j}$ over all $1
  \le i \le k$ and $j \in I_i$, we obtain
  \[
    F_M(m_i) - m_i \in (t^{>0}) M^\circ
  \]
  modulo $J\dparen[t]$. Defining $C = A / J$, so that $A \twoheadrightarrow C
  \twoheadrightarrow B$, we conclude that the Novikov isocrystal
  \[
    (M \otimes_{A\dparen[t]} C\dparen[t], F_M \otimes F)
  \]
  has the property that the $C\dbrack[t]$-submodule of $M \otimes_{A\dparen[t]}
  C\dparen[t]$ generated by the images of $m_1, \dotsc, m_k$ satisfies the
  hypotheses of Proposition~\ref{Prop:LatticeIsocrystal}. This shows that the
  base change of $(M, F_M)$ to $C$ lies in the essential image of $\vect(C) \to
  \novisoc(C)$, and by Proposition~\ref{Prop:NovToIsocFullyFaithful}, it also
  follows that the base change of $(M, \varphi)$ to $C$ lies in the essential
  image of $\vect(C) \to \novdd(C)$.

  On the other hand, because all elements of $J$ are nilpotent and $J$ is
  finitely generated, some finite power of $J$ is the zero ideal. This allows us
  to iteratively apply Lemma~\ref{Lem:NovDescentSquareZero} to conclude that
  $(M, \varphi)$ lies in the essential image over $\vect(A) \to \novdd(A)$.
\end{proof}

\subsection{Restricting the exponents} \label{Subsec:Exponent}
\def\novexp{\Gamma}

We finally remove the hypothesis that $\Gamma = \mathbb{R}$. Note there is a
natural functor $\novdd(A) \to \novr{\novdd(A)}$ given by tensoring along the
ring homomorphism $A\dparen[t] \to \novr{A\dparen[t]}$.

\begin{theorem} \label{Thm:NovDescent}
  Let $A$ be a commutative ring, and let $\Gamma \subseteq \mathbb{R}$ be an
  additive submonoid. Then the functor $\vect(A) \to \novdd(A)$ is an
  equivalence of categories.
\end{theorem}

\begin{proof}
  By Proposition~\ref{Prop:NovDescentFullyFaithful}, it is enough to prove
  essential surjectivity. Let $(M, \varphi) \in \novdd(A)$ be a Novikov descent
  datum. By Proposition~\ref{Prop:NovDescentR}, there exists a finite projective
  $A$-module $M_0$ together with an isomorphism
  \[
    f \colon (M \otimes_{A\dparen[t]} \novr{A\dparen[t]}, \varphi
    \otimes_{A\dparen[t][u]} \novr{A\dparen[t][u]}) \xrightarrow{\cong}
    (\novr{M_0\dparen[t]}, \id_{M_0} \otimes \novr{A\dparen[t][u]})
  \]
  in $\novr{\novdd(A)}$. This in particular allows us to regard $M$ as an
  $A\dparen[t]$-submodule of $\novr{M_0\dparen[t]}$ with the property that the
  images of
  \[
    M \otimes_{A\dparen[t], \pi_1^\ast} A\dparen[t][u] \hookrightarrow
    \novr{M_0\dparen[t][u]}, \quad M \otimes_{A\dparen[t], \pi_2^\ast}
    A\dparen[t][u] \hookrightarrow \novr{M_0\dparen[t][u]}
  \]
  agree.

  Consider any element $x$ in the common image. Since it can be written as a
  finite sum
  \[
    x = \sum_{i=1}^n (\pi_1^\ast m_i) a_i, \quad m_i \in M, \quad a_i \in
    A\dparen[t][u],
  \]
  we see that $\supp(x)$ must be contained in $\mathbb{R} \times \Gamma$. For
  every element $m \in M$, we in particular have $\supp(\pi_2^\ast m) \subseteq
  \mathbb{R} \times \Gamma$, and hence $\supp(m) \subseteq \Gamma$. This shows
  that
  \[
    M \subseteq M_0\dparen[t],
  \]
  and it now suffices to prove that this inclusion is an equality.

  We achieve this by running the same argument for the dual descent datum. Write
  $g \colon M \hookrightarrow M_0\dparen[t]$ for the inclusion above. If we
  consider the linear dual of $f$, we obtain an isomorphism
  \[
    (f^\vee)^{-1} \colon M^\vee \otimes_{A\dparen[t]} \novr{A\dparen[t]}
    \xrightarrow{\cong} \novr{M_0^\vee\dparen[t]}
  \]
  compatible with descent data. By the argument above, there exists an inclusion
  $h \colon M^\vee \hookrightarrow M_0^\vee\dparen[t]$ that recovers
  $(f^\vee)^{-1}$ upon base changing along $A\dparen[t] \to \novr{A\dparen[t]}$.
  This shows that the composition
  \[
    M_0\dparen[t] \xrightarrow{h^\vee} M \xrightarrow{g} M_0\dparen[t]
  \]
  is the identity map upon base changing along $A\dparen[t] \to
  \novr{A\dparen[t]}$. Because both sides are finite projective and $A\dparen[t]
  \hookrightarrow \novr{A\dparen[t]}$ is injective, the composition $g \circ
  h^\vee$ is the identity map. In particular, $g$ is surjective, and therefore a
  bijection.
\end{proof}

\section{Application to perfectoid geometry} \label{Sec:Perfectoid}
\def\novexp{{1/p^\infty}}

We now record some consequences of our descent result. We will use
Theorem~\ref{Thm:Main} with the exponent submonoid $\Gamma =
\mathbb{Z}[p^{-1}]$. It is customary to write
\[
  {\def\novexp{\Gamma} A\dparen[t]} = A\dparen[t]
\]
in the setting of perfectoid geometry.

Let $\mathsf{Perf}$ be the category of affinoid perfectoid spaces in
characteristic $p$, equipped with the v-topology. For a perfect
$\mathbb{F}_p$-algebra $A$, we consider the v-sheaf
\[
  \Spd A \colon \mathsf{Perf}^\mathrm{op} \to \mathsf{Set}; \quad \Spa(R, R^+)
  \mapsto \Hom(A, R^+).
\]
Similarly for $B$ a perfect (hence perfectoid) complete Tate
$\mathbb{F}_p$-algebra, we define the v-sheaf
\[
  \Spd B = \Spd(B, B^\circ) \colon \mathsf{Perf}^\mathrm{op} \to \mathsf{Set};
  \quad \Spa(R, R^+) \mapsto \Hom_\mathrm{cts}((B, B^\circ), (R, R^+)).
\]
We refer the reader to \cite{SW20} for the detailed
definitions of the objects appearing in this section.

\begin{definition}
  For a v-sheaf $X$, we define a \textdef{vector bundle} on $X$ to be an
  association to each $(R, R^+)$-point $\xi \in X(R, R^+)$ a locally finite free
  sheaf $\mathscr{V}_\xi$ on $\Spa(R, R^+)$, together with the data of, for each
  map $f \colon \Spa(S, S^+) \to \Spa(R, R^+)$ and $\xi \in X(R, R^+)$, an
  isomorphism $f^\ast \mathscr{V}_\xi \cong \mathscr{V}_{f^\ast \xi}$,
  satisfying the condition that for every
  \[
    \Spa(T, T^+) \xrightarrow{g} \Spa(S, S^+) \xrightarrow{f} \Spa(R, R^+)
  \]
  and $\xi \in X(R, R^+)$, the three identifications between $g^\ast f^\ast
  \mathscr{V}_\xi$, $g^\ast \mathscr{V}_{f^\ast \xi}$, and $\mathscr{V}_{g^\ast
  f^\ast \xi}$ commute.
\end{definition}

Let $A$ be a discrete perfect $\mathbb{F}_p$-algebra. There is a v-cover
\[
  X = \Spd(A\dparen[t], A\dbrack[t]) \to \Spd(A, A) = Y
\]
of $Y$ be an affinoid perfectoid space $X$. We note that $X$ represents the
functor sending a perfectoid Huber pair $(A, A) \to (R, R^+)$ to a topologically
nilpotent unit $\varpi \in R^{\circ\circ}$, because uniformity and perfectness
of $R$ implies that there exists a unique continuous map $A\dparen[t] \to R$
sending $t \mapsto \varpi$. This means that both $X \times_Y X$ and $X \times_Y
X \times_Y X$ are representable by perfectoid spaces
\begin{align*}
  X \times_Y X &= \Spa A\dbrack[t][u] \setminus \lbrace \lvert tu \rvert = 0
  \rbrace \\ &= \bigcup_{n=2}^\infty \Spd A\dbrack[t][u]\Bigl\langle
  \frac{t^n}{u}, \frac{u^n}{t} \Bigr\rangle[1/tu], \\ X \times_Y X \times_Y X &=
  \Spa A\dbrack[t][u][v] \setminus \lbrace \lvert tuv \rvert = 0 \rbrace \\ &=
  \bigcup_{n=3}^\infty \Spd A\dbrack[t][u][v] \Bigl\langle \frac{t^n}{uv},
  \frac{u^n}{tv}, \frac{v^n}{tu} \Bigr\rangle[1/tuv].
\end{align*}

\begin{lemma} \label{Lem:GlobalSection}
  We have
  \begin{align*}
    H^0(X \times_Y X, \mathscr{O}_{X \times_Y X}) &= A\dparen[t][u], \\ H^0(X
    \times_Y X \times_Y X, \mathscr{O}_{X \times_Y X \times_Y X}) &=
    A\dparen[t][u][v].
  \end{align*}
\end{lemma}

\begin{proof}
  For the first equality, we note that we can write
  \[
    A\dbrack[t][u]\Bigl\langle \frac{t^n}{u}, \frac{u^n}{t} \Bigr\rangle
    \Bigl[\frac{1}{tu}\Bigr] = \left\lbrace \sum_{i \in I} a_i t^{d_i} u^{e_i} :
    \begin{matrix} a_i \in A, d_i, e_i \in \mathbb{Z}[p^{-1}], \\ \lbrace i \in
    I : \alpha d_i + e_i \lt C \rbrace \text{ is finite} \\ \text{for all } C
    \in \mathbb{R} \text{ and } n^{-1} \le \alpha \le n \end{matrix}
    \right\rbrace.
  \]
  As we increase the $n$ and take the intersection of the rings, we see that we
  allow all possible values of $\alpha > 0$, and hence the ring of global
  sections of $X \times_Y X$ is $A\dparen[t][u]$. The proof of the second
  equality is similar.
\end{proof}

\begin{corollary} \label{Cor:VectorBundleOnSpd}
  Let $A$ be a discrete perfect $\mathbb{F}_p$-algebra. Then the natural functor
  \[
    \lbrace \text{vector bundles on } \Spec A \rbrace \to \lbrace \text{vector
    bundles on } \Spd A \rbrace
  \]
  is an exact tensor equivalence.
\end{corollary}

\begin{proof}
  We use the v-cover
  \[
    X = \Spd(A\dparen[t], A\dbrack[t]) \to \Spd A = Y
  \]
  from above. By \cite[Lemma~17.1.8]{SW20} the category of vector bundles on
  $\Spd A$ can be identified with the category of pairs $(\mathscr{V},
  \varphi)$, where $\mathscr{V}$ is a vector bundle on $X$ and $\varphi$ is an
  isomorphism of vector bundles over $X \times_Y X$ satisfying a cocycle
  condition over $X \times_Y X \times_Y X$.
  By \cite[Theorem~2.7.7]{KL15} and again
  \cite[Lemma~17.1.8]{SW20}, the vector bundle $\mathscr{V}$
  corresponds to a finite projective module $M$ over $A\dparen[t]$.

  By Lemma~\ref{Lem:GlobalSection}, the natural functor from the category of
  finite projective $A\dparen[t][u]$-modules to the category of vector bundles
  on $X \times_Y X$ is fully faithful. If we denote by $\pi_1, \pi_2 \colon X
  \times_Y X \to X$ the two projection maps, then the descent datum $\varphi
  \colon \pi_1^\ast \mathscr{V} \cong \pi_2^\ast \mathscr{V}$ corresponds to an
  isomorphism of $A\dparen[t][u]$-modules
  \[
    \varphi \colon M \otimes_{A\dparen[t], \pi_1^\ast} A\dparen[t][u]
    \xrightarrow{\cong} M \otimes_{A\dparen[t], \pi_2^\ast} A\dparen[t][u].
  \]
  Finally the cocycle condition is defined on $X \times_Y X \times_Y X$, and
  hence by Lemma~\ref{Lem:GlobalSection} can be checked after tensoring to
  \[
    H^0(X \times_Y X \times_Y X, \mathscr{O}_{X \times_Y X \times_Y X}) =
    A\dparen[t][u][v].
  \]
  This shows that a vector bundle on $\Spd A$ corresponds to a Novikov descent
  datum for $A$ with $\Gamma = \mathbb{Z}[p^{-1}]$. Applying
  Theorem~\ref{Thm:Main} now gives the desired result.
\end{proof}

We now compare the finite \'{e}tale sites of $\Spec A$ and $\Spd A$.

\begin{definition}[{\cite[Definition~10.1]{Sch22}}]
  We say that a map $f \colon X \to Y$ of v-sheaves is \textdef{finite
  \'{e}tale} when for every map $\Spa(R, R^+) \to Y$, the fiber product $\Spa(R,
  R^+) \times_Y X$ is representable by an affinoid perfectoid $\Spa(S, S^+)$
  where $R \to S$ is finite \'{e}tale and $S^+$ is the integral closure of the
  image of $R^+$.
\end{definition}

\begin{corollary} \label{Cor:FiniteEtaleOverSpd}
  Let $A$ be a perfect $\mathbb{F}_p$-algebra. Then the natural functor
  \[
    \lbrace \text{schemes finite \'{e}tale over } \Spec A \rbrace \to \lbrace
    \text{v-sheaves finite \'{e}tale over } \Spd A \rbrace
  \]
  is an equivalence.
\end{corollary}

\begin{proof}
  We note that a finite \'{e}tale $R$-algebra is necessarily a finite projective
  $R$-module. By Corollary~\ref{Cor:VectorBundleOnSpd}, we have an equivalence
  between the category of vector bundles on $\Spec A$ with an algebra structure
  and the category of vector bundles on $\Spd A$ with an algebra structure. It
  remains to check that the condition of being \'{e}tale agrees on both sides.

  If the finite $A$-algebra $B$ is \'{e}tale, it is clear that the corresponding
  $R$-algebra $R \otimes_A B$ is also \'{e}tale over $R$ for every $A \to R^+$.
  In the other direction we note that if $A\dparen[t] \to B \otimes_A
  A\dparen[t]$ is finite \'{e}tale, then so is $A \to B$ by
  \cite[Theorem~08XE]{Sta24}, as the inclusion $A \to A\dparen[t]$
  splits as a map of $A$-modules.
\end{proof}

\bibliographystyle{amsalpha}
\bibliography{references}

\end{document}